\documentclass{amsart}

\RequirePackage{amspaperstyle}

\newcommand{\classH}{\mathcal{H}}

\newcommand{\minht}{\mathrm{minht}}
\newcommand{\Sob}{\mathcal{S}}

\newcommand{\ugen}{\mathrm{z}}
\newcommand{\dioeps}{\psi} 

\newcommand{\wholenorm}[1]{\mathbf{N}_{#1}}

\newcommand{\wedgespace}[1]{\mathbf{V}_{#1}}
\newcommand{\wedgevec}[1]{\mathbf{v}_{#1}}
\newcommand{\wedgerep}[1]{\rho_{#1}}
\newcommand{\wedgeorbit}[1]{\eta_{#1}}

\newcommand{\compactcont}{C_c}
\newcommand{\compactsmooth}{C_c^\infty}

\newcommand{\minobs}{\vartheta}

\newcommand{\dimemb}{N}

\newcommand{\kappaundistcompl}{{\kappa_1}}
\newcommand{\kappainjradius}{{\kappa_2}}
\newcommand{\kappaeffgen}{{\kappa_3}}
\newcommand{\kappaeffav}{{\kappa_{4}}}

\newcommand{\kappacomm}{\kappa_{6}}

\newcommand{\kappaalminvG}{\kappa_{9}}

\newcommand{\Adiam}{\mathrm{A}_{1}}
\newcommand{\ASobcont}{{\mathrm{A}_2}}
\newcommand{\AeffavT}{{\mathrm{A}_{3}}}
\newcommand{\AeffavcorH}{{\mathrm{A}_{4}}}
\newcommand{\Adiocond}{\AeffavcorH}
\newcommand{\Aeffclos}{{\mathrm{A}_{5}}}
\newcommand{\Aeffclosres}{{\mathrm{A}_{6}}}
\newcommand{\Aeffcloscor}{{\mathrm{A}_{7}}}

\newcommand{\Atransres}{\mathrm{A}_{9}}

\newcommand{\kappaSobineq}{{d_0}}

\newcommand{\Cdiam}{{\mathrm{C}_{1}}}
\newcommand{\Ceffav}{{\mathrm{C}_{2}}}
\newcommand{\Ceffavgeom}{{\mathrm{C}_{3}}}
\newcommand{\CeffavcorH}{\mathrm{C}_{4}}
\newcommand{\Ceffcloscor}{\mathrm{C}_{5}}
\newcommand{\Ctranspts}{{\mathrm{C}_{6}}}

\newcounter{C}

\newcommand{\tempered}{K}

\begin{document}

\title[Effective equidistribution of orbits under semisimple groups]{Effective equidistribution of orbits under semisimple groups on congruence quotients}

\author{Andreas Wieser}
\address{Department of Mathematics, Princeton University, Fine Hall, Washington Road, Princeton, NJ 08540, USA}
\email{aw2638@math.princeton.edu}
\thanks{This project was supported by the ERC grant HomDyn, ID 833423 as well as the SNF Doc. Mobility grant 195737.}

\date{\today}

\begin{abstract}
We prove an effective equidistribution result for periodic orbits of semisimple groups on congruence quotients of an ambient semisimple group. 
This extends a previous work of Einsiedler, Margulis and Venkatesh. 
The main new feature is that we allow for periodic orbits of semisimple groups with nontrivial centralizer in the ambient group. 
Our proof uses crucially an effective closing lemma from work of the author with Lindenstrauss, Margulis, Mohammadi, and Shah.
\end{abstract}

\maketitle

\section{Introduction}

Let $G$ be a Lie group and let $\Gamma < G$ be a lattice.
In her seminal work, Ratner \cite{ratner91-measure} 
classified probability measures on $X = \lquot{\Gamma}{G}$ invariant and ergodic under a one-parameter unipotent subgroup of $G$.
The classification asserts that any such measure is \emph{homogeneous} (or \emph{algebraic}), i.e.~it is the invariant probability measure $\mu_{xL}$ on a periodic orbit $xL$ of a closed subgroup $L \subset G$.
Subsequent work by Mozes and Shah \cite{MozesShah} established the same type of rigidity for limits of such measures using the linearization technique by Dani and Margulis~\cite{DaniMargulis93}.

\begin{theorem}[{\cite{MozesShah}}]
Let $\mu_i$ be a sequence of probability measures on $X$ so that each $\mu_i$ is invariant and ergodic under a one-parameter unipotent subgroup.
Suppose that $\mu_i\to \mu$ for a probability measure $\mu$ in the weak${}^\ast$-topology. Then $\mu$ is homogeneous.
\end{theorem}

Giving quantitative equidistribution results in homogeneous dynamics, in particular in the context of orbits of semisimple groups, has been a major theme in recent years, and appears prominently in Margulis' list of open problems in homogeneous dynamics \cite[\S1.3]{Margulis-Problems}. In particular the problem of quantifying the Mozes-Shah Theorem has attracted considerable attention.
The progress is arguably most complete for orbits of horospherical groups. 
Here, Margulis' thickening technique \cite{Margulis-Thesis} paired with effective mixing for diagonalizable flows can give a polynomially effective rate --- see for example \cite{KleinbockMargulis-Thickening}.
Other known instances with effective rates include the equidistribution of Hecke points (see e.g.~\cite{ClozelOhUllmo}) and of translates of a fixed orbit of a symmetric group (see e.g.~\cite{BenoistOh}).

In this article, we consider periodic orbits of semisimple groups in the following arithmetic setup:
\begin{itemize}
\item $G = \G(\R)$ where $\G$ is a connected semisimple $\Q$-group, and $G^+$ the identity component of $G$.
\item $\Gamma \subset \G(\Q)$ is a congruence subgroup.
\item $H < G$ is a connected semisimple Lie subgroup without compact factors.
\end{itemize}
We now turn to phrase the main theorem of the current article; important preceeding results will be discussed below in \S\ref{sec:history}.
We denote by $\classH$ the family of connected $\Q$-subgroups of $\G$ whose radical is unipotent.
By a theorem of Borel and Harish-Chandra \cite{BorelHarishChandra}, for any $\BM\in \mathcal{H}$ the orbit $\Gamma \BM(\R)$ is periodic i.e.~$\Gamma \cap \BM(\R)$ is a lattice in $\BM(\R)$. 
Translates of these orbits with small volume form natural obstacles to equidistribution of periodic $H$-orbits.
Instead of the volume, we shall use its arithmetic counterpart -- the discriminant $\disc(\cdot)$ -- to measure the complexity of orbits; see \S\ref{sec:discriminant}.
When $X$ is noncompact, one can construct explicitly (see \S\ref{sec:sob norms}) a proper map $\height :X \to [0,\infty)$; for a subset $B \subset X$ we write $\minht(B) = \inf\{\height(x): x\in B\}$ (note that this height has nothing to do with any arithmetic height which is closely related to our notion of discriminant).

The following is our main theorem; the main innovation is a different, general treatment of the centralizer of $H$.

\begin{theorem}[Effective equidistribution of semisimple orbits]\label{thm:main}
There exist $\delta >0$ and $d \geq 1$ depending on $G$ and $H$ with the following properties.
Let $xH = \Gamma g H$ be a periodic orbit in $X$ and let $\mu_{xH}$ be the Haar probability measure on $xH$. 
Suppose that $D >0$ is such that
\begin{align}\label{eq:diophantineassump thm}
\disc(\Gamma \BM(\R)g) \geq D
\end{align}
for all orbits $\Gamma \BM(\R) g$ containing $xH$ where $\BM \in \classH$ is a proper $\Q$-subgroup of $\G$.
Then
\begin{align*}
\Big| \int_{xH} f \de \mu_{xH} - \int_{xG^+} f \de \mu_{xG^+}\Big| 
\ll_{G,H,\Gamma} D^{-\delta}\minht(xH) \Sob_d(f)
\end{align*}
where $\Sob_d$ is an $L^2$-Sobolev norm of degree $d$ (cf.~\S\ref{sec:sob norms}).
\end{theorem}

For a given (semisimple) $\Q$-subgroup $\mathbf{H}$ of $\G$ it is often difficult to classify all intermediate $\Q$-subgroups $\mathbf{H} \subset \BM \subset \G$.
However, in view of the rate in Theorem~\ref{thm:main} it is usually sufficient to verify \eqref{eq:diophantineassump thm} for a suitable subcollection of $\classH$. We illustrate this for orthogonal groups in Section~\ref{sec:intro-orth} below and give a more general statement of this kind in Theorem~\ref{thm:exhaustive}.

\subsection{Some context}\label{sec:history}

Known effective equidistribution results for periodic $H$-orbits (in the arithmetic setup of this article) heavily rely on the uniformity of the spectral gap for $H$-representations $L^2_0(xH)$ where $xH$ runs over all \emph{periodic} $H$-orbits in $X$.
This is a deep result (property $(\tau)$) due to various authors, among others Selberg \cite{Selberg} and Jacquet-Langlands \cite{JacquetLanglands} (for groups of type $A_1$), Kazhdan \cite{Kazhdan-PropertyT}, Burger-Sarnak \cite{BurgerSarnak}, and Clozel \cite{Clozel-PropertyTau}; we refer to \cite[\S4.1]{emmv} for a more detailed account of the history.

Using uniformity of the spectral gap, breakthrough work of Einsiedler, Margulis, and Venkatesh \cite{emv} established the following effective equidistribution result for periodic $H$-orbits in $X$.

\begin{theorem}[\cite{emv}]\label{thm:emv}
Suppose that the centralizer of $H$ in $G$ is finite.
Then there exists $\delta >0$ and $d\geq 1$ depending on $G$ and $H$ and $V_0>0$ depending on $G$, $H$, and $\Gamma$ with the following properties.

Let $xH\subset X$ be a periodic $H$-orbit.
For any $V \geq V_0$ there exists a connected intermediate group $H \subset S \subset G$ for which $xS$ is periodic with $\vol(xS) \leq V$ and so that for any $f \in \compactsmooth(X)$
\begin{align*}
\Big| \int_{xH} f \de \mu_{xH} - \int_{xS} f \de \mu_{xS}\Big| 
< V^{-\delta} \Sob_d(f).
\end{align*}
\end{theorem}

In practice, one can take $V$ to be a sufficiently small power of the volume of $xH$ to find a well-approximated orbit of a larger intermediate group.
The exponent $\delta>0$ is in principle computable, see \cite{Mohammadi-explicitconst} for the special case $\SO(2,1)< \SL_3(\R)$.
We also refer to the surveys of Einsiedler \cite{Einsiedler-effectiveintro} and Einsiedler, Mohammadi \cite{EinsiedlerMohammadi-effectivesurvey}.

While the centralizer assumption does not seem crucial to the method, it is intricately built into the proof of Theorem~\ref{thm:emv}.
It does assert among other things that periodic $H$-orbits appear `discretely'. 
Indeed, for any periodic orbit $xH$ and any $c\in G$ centralizing $H$ the orbit $xcH$ is also periodic and hence periodic $H$-orbits appear in families parametrized by the centralizer.
Moreover, the centralizer assumption restricts the possibilities for intermediate subgroups (cf.~\cite[App.~A]{emv}): there are finitely many and all are semisimple.

In view of applications, one would like the centralizer assumption to be removed.
Previously, this had been achieved in the following special cases:
\begin{itemize}
\item In \cite{emvforSld}, Aka, Einsiedler, Li, and Mohammadi established the analogue of Theorem~\ref{thm:emv} for $\G= \SL_d$ (or a real split $\Q$-form of $\SL_d$) and
\begin{align*}
H = \Big\{\begin{pmatrix}
g_1 & 0 \\ 0 & g_2
\end{pmatrix}: g_1 \in \SL_{k}(\R),g_2 \in \SL_{d-k}(\R)\Big\}
\end{align*}
for $(k,d)\neq (2,4)$. In this case, the centralizer of $H$ is a one-dimensional (split) torus. In particular, an $H$-orbit can lie far up in the cusp, which is why the statement as in Theorem~\ref{thm:emv} needs to be adapted by a suitable height function (as in Theorem~\ref{thm:main}).
The case $(k,d)= (2,4)$ is ruled out due to the presence of an intermediate symplectic group.
\item In \cite{EinsiedlerWirth}, Einsiedler and Wirth established the analogue of Theorem~\ref{thm:emv} for $\G = \SO_Q$ where $Q$ is a rational quadratic form of signature $(n,1)$ and $H \simeq \SO(2,1)^+$.
In this case, the centralizer is compact (it is isomorphic to $\SO(n-2)$). 
\end{itemize}
The above theorems could also be phrased as equidistribution theorems in the ambient space with a rate polynomial in the minimal intermediate volume (in analogy to Theorem~\ref{thm:main}).
The current article removes the centralizer assumption in this phrasing (see also Remark~\ref{rem:extensions} below).
This heavily relies on an effective closing lemma from work of the author with Lindenstrauss, Margulis, Mohammadi, and Shah \cite{EffClosing} (see Theorem~\ref{thm:effectiveclosing} below).
 
The rate given in any equidistribution result as in Theorems~\ref{thm:main} or~\ref{thm:emv} can only be as fast as the minimal volume or discriminant of an intermediate orbit. 
In that sense, the above theorems are optimal.
The decay exponents are likely far from optimal in all of the above theorems.

\begin{remark}[Extensions]\label{rem:extensions}
In principle, the techniques of this paper ought to give a similar statement as Theorem~\ref{thm:emv} (the main theorem of \cite{emv}).
Indeed, one might imagine applying Theorem~\ref{thm:main} by `downward induction': if $D >0$ does not satisfy \eqref{eq:diophantineassump thm} for some $\mathbf{M}$, switch to the equidistribution problem for $xH \subset \Gamma \BM(\R) g$. Difficulties in doing so arise for example from the fact that $\BM$ might not be semisimple or from the (here inexplicit) dependency of the implicit constants on $G$ and $\Gamma$.
\end{remark}

\begin{remark}[Semisimple adelic periods]\label{rem:adelic}
Theorem~\ref{thm:emv} has been extended to adelic periods by Einsiedler, Margulis, Mohammadi, and Venkatesh \cite{emmv} for maximal subgroups and by Einsiedler, R\"uhr, and Wirth \cite{ERW16} in a special case for a  partial resolution of a conjecture of Aka, Einsiedler, and Shapira \cite{AES14}.
The former clarifies the dependency of Theorem~\ref{thm:emv} on $H$.
Lastly, we remark that a full resolution of the effective equidistribution problem for semisimple adelic periods promises an abundance of interesting arithmetic applications. This includes an effective version of the integral Hasse principle for representations of quadratic forms (cf.~\cite{localglobalEV}).
\end{remark}

\begin{remark}
The currently available effective results on unipotent flows do not require orbits to be periodic. 
Effective analogues of Ratner's equidistribution theorem \cite{ratner91-topological} are known in various instances.
As mentioned earlier, when the acting unipotent group is horospherical equidistribution is well understood -- see for instance \cite{Sarnak-horocycle,Burger-horocycle,KleinbockMargulis-Thickening,FlaminioForni,Strombergsson13,SarnakUbis,KleinbockMargulis-effective,McAdam-primetimes,Katz-disjointness}.
In a somewhat different direction, effective equidistribution results have been established for the horospherical subgroup of $\SL_n(\R)$ acting on arithmetic quotients of $\SL_n(\R) \rtimes (\R^n)^k$ \cite{Strombergsson15,BrowningVinogradov,StrombergssonVishe,kim2024effective}.
In a recent breakthrough, Lindenstrauss and Mohammadi \cite{LindenstraussMohammadi} and Lindenstrauss, Mohammadi, and Z.~Wang \cite{LMW-effequi1,LMW-effequi2} proved effective equidistribution for one-parameter unipotent flows on arithmetic quotients of $\SL_2(\R) \times \SL_2(\R)$ or $\SL_2(\C)$; for quotients of rank two real split groups see \cite{Lei-SL3,LMWY}.
We refer to \cite[\S1.4]{LMW-effequi2} for a more detailed account of the history.
\end{remark}

\subsection{A special case: Special orthogonal groups}\label{sec:intro-orth}
To illustrate Theorem~\ref{thm:main}, we give explicit conditions to verify \eqref{eq:diophantineassump thm} in a special case related to \cite{EinsiedlerWirth}.
For the remainder of the section, consider the semisimple $\Q$-group $\G = \SO_Q< \SL_d$ where $Q$ is an indefinite rational quadratic form of signature $(p,q)$ where $p+q=d\geq 4$ and $p,q>0$.
We let $\Gamma < \SO_Q(\Q) \cap \SL_d(\Z)$ be a congruence subgroup and $G = \G(\R)$, $X = \lquot{\Gamma}{G}$ as before.
As acting group, take $H$ to be the identity component of the pointwise stabilizer subgroup of a positive definite real subspace of $\R^d$ of dimension less than $p$.
In particular, $H \simeq \SO(p',q)^+$ for some $p'<p$ and the centralizer of $H$ is a compact group isomorphic to $\SO(p-p')$.

For any collection of vectors $\mathcal{B} \subset \Q^d$ we set
\begin{align*}
\BM_{\mathcal{B}}
= \{g \in \G: g.v = v \text{ for all } v \in \mathcal{B}\}.
\end{align*}
Moreover, we define for any periodic orbit $xH = \Gamma g H$
\begin{align*}
\minobs'(xH) = \inf\{\disc(\Gamma \BM_v(\R) g): v \in \Q^d \text{ with } Q(v) \neq 0 \text{ and } \Gamma \BM_v(\R) g \supset xH\}.
\end{align*}
Note that many periodic intermediate orbits as in Theorem~\ref{thm:main} are not of the above form $\Gamma \BM_v(\R) g$.
Nevertheless, Theorem~\ref{thm:main} together with an analysis of intermediate groups yields the following theorem.

\begin{theorem}\label{thm:effequi orth}
There exist $\delta >0$ and $d \geq 1$ depending on $G$ and $H$ with the following properties.
For any periodic $H$-orbit $xH$ and any $f \in \compactsmooth(X)$
\begin{align*}
\Big| \int_{xH} f \de \mu_{xH} - \int_{xG^+} f \de \mu_{xG^+}\Big| 
\ll_{G,H,\Gamma} \minobs'(xH)^{-\delta} \Sob_d(f).
\end{align*}
\end{theorem}

Theorem~\ref{thm:effequi orth} can be seen as an extension of the effective equidistribution result of Einsiedler and Wirth \cite{EinsiedlerWirth} (where $q=1$) mentioned earlier.
Orbits of special orthogonal subgroups are of particular interest in view of number theoretic applications, most notably towards the integral Hasse principle -- see \cite{localglobalEV} and Remark~\ref{rem:adelic}.
As the spin group has strong approximation for indefinite forms, the integral Hasse principle for indefinite forms is readily obtained.
(In particular, the current result in Theorem~\ref{thm:effequi orth} does not yield any progress towards it.)
However, Theorem~\ref{thm:effequi orth} does potentially yield equidistribution results in the spirit of \cite{AES14,AES15,AEW-2in4,AMW-higherdim} for representations of positive-definite integral quadratic forms by indefinite integral quadratic forms.

\begin{remark}\label{rem:exptheta}
One can relate $\minobs'(xH)$ to the length of the shortest integer vector in a certain rational subspace as follows. By the Borel-Wang density Theorem \cite[Ch.~II]{Margulis-book} there exists for any periodic $H$-orbit $xH= \Gamma gH$ a rational subspace $L \subset \Q^d$ (positive-definite over $\R$) such that $gHg^{-1} = \BM_L(\R)^+$.
The subspace $L$ depends not only on $xH$ but also on the representative $g$. 
Since $H$ has a compact centralizer in $G$, classical non-divergence results thus imply that $\minht(\cdot)$ is uniformly bounded on periodic $H$-orbits.
Equivalently, any periodic orbit $xH = \Gamma g H$ has a representative $g$ of absolutely bounded size.
Defining the above subspace $L$ with this choice, we have $\norm{v}^\star \gg\disc(\Gamma \BM_v(\R) g) \gg \norm{v}^\star$ for any primitive integral vector $v \in L \setminus \{0\}$ (if $\Gamma \BM_v(\R) g$ contains $xH$, $v$ must be contained in $L$).
In other words, Theorem~\ref{thm:main} provides a rate polynomial in $\min_{0 \neq v \in L} \norm{v}$.
\end{remark}

\subsection{Some ideas of the proof of Theorem~\ref{thm:main}}
We follow an a priori familiar strategy already used in previous works \cite{emv,emvforSld,EinsiedlerWirth} which in essence effectivizes a proof of Ratner's measure classification theorem for $H$-invariant and ergodic measures on $X$ (cf.~\cite[\S2]{emv}, \cite{Einsiedler-Ratner}).
The outline below recovers this strategy. We omit exact definitions and some (partially important) details as well as polynomial dependencies where convenient in order to focus on the ideas.

Suppose that $\mu$ is the Haar probability measure on a periodic $H$-orbit $xH$.
The proof proceeds by recursion showing that $\mu$ is `almost invariant' under larger and larger intermediate Lie algebras $\mathfrak{h}\subset \Fs\subset \Fg$.
Here, no restriction whatsoever is imposed on $\Fs$ (e.g.~$\Fs$ might have a center).
A measure $\mu$ is $\varepsilon$-almost invariant under $\Fs$ if the $\exp(Z)$-translate of $\mu$ is within $\varepsilon$ of $\mu$ (in a suitable sense -- see Definition~\ref{def:almost invariant}) for all $Z \in \Fs$ with $\norm{Z}\leq 1$.
If the measure $\mu$ is $\varepsilon$-almost invariant under $\Fg$, the spectral gap on $xG^+$ implies the theorem.
Moreover, at each step of the recursion $\varepsilon$ will be polynomial in $D$.

Assume from now on that $\mu$ is $\varepsilon$-almost invariant under some intermediate Lie algebra $\Fs$. As $H$ is semisimple, there exists an $H$-invariant complement $\Fr$ to $\Fs$. We further fix a one-parameter unipotent subgroup $U = \{u_t: t\in \R\}$ of $H$ which acts ergodically on $xH$.

\subsubsection{From transversal generic points to almost invariance}\label{sec:intro-addinv}
Using effective decay of matrix coefficients for the $H$-representation $L_0^2(\mu)$, Einsiedler, Margulis, and Venkatesh \cite{emv} have established the existence of a large set of $U$-generic points where `generic' is taken to be an effectivization (cf.~Definition~\ref{def:discrepancy}) of the usual Birkhoff genericity notion.
For the purposes of this outline, assume now that there are two points $x_1,x_2 \in X$ with $x_2 = x_1 \exp(r)$ for some nonzero small $r \in \Fr$ and with
\begin{align}\label{eq:gen-pseudo}
\Big|\frac{1}{\sqrt{n}}\int_{n}^{n+\sqrt{n}} f(x_iu_t) \de t - \int f \de \mu \Big| \leq C(f) \frac{1}{\sqrt{n}}
\end{align}
for all $1 \leq n \leq \varepsilon^{-\delta}$ where $C(f)>0$ depends on $f \in \compactsmooth(X)$ and $\delta>0$ is given.
The behaviour of the displacement between $x_1u_t$ and $x_2u_t$ in the time $t$ is governed by the polynomial map $p:t \mapsto \Ad_{u_{-t}}(r)-r $. One distinguishes two situations:
\begin{enumerate}
\item When the supremum of $\norm{p(t)}$ over $t \in [0,\varepsilon^{-\delta}]$ is, say, at least $1$, then polynomial divergence techniques show that the pieces of the $U$-orbits are roughly parallel towards the end. In this case, one obtains that $\mu$ is $\varepsilon^\star$-almost invariant under that displacement.
\item When the supremum is at most $1$, the polynomial $\Ad_{u_{-t}}(r)$ is roughly constant equal to $r$ on a polynomially shorter interval (such as $[0,\varepsilon^{-\delta/2}]$).
In this case, $\mu$ is $\varepsilon^\star$-almost invariant under $r$.
\end{enumerate}
In both of the above cases, the size of $r$ needs to be controlled in terms of $\varepsilon$.
For instance, if the size $\norm{r}$ is a lot smaller than $\varepsilon$, the conclusion of (2) is vacuous by Lipschitz continuity. In short, the displacement $r \in \Fr$ needs to satisfy
\begin{align}\label{eq:pseudoexpdispl}
\varepsilon^{\star}\leq \norm{r} \leq \varepsilon^\star
\end{align}
with suitable exponents from the start.

Lastly, we remark that the effective generation results established in \cite{emv} also apply here to show that if $\mu$ is $\varepsilon$-almost invariant under $\Fs$ and $\varepsilon^\star$-almost invariant under some unit vector in $\Fr$, it is $\varepsilon^\star$-almost invariant under a new intermediate Lie algebra of larger dimension.
Hence, the recursion may be continued.

\subsubsection{Existence of transversal generic points}\label{sec:intro-transv}

Overall, it remains to discuss why two points $x_1,x_2$ with the Birkhoff genericity assumption in \eqref{eq:gen-pseudo} and a controlled transversal displacement as in \eqref{eq:pseudoexpdispl} exist (see Proposition~\ref{thm:trans2} for an exact statement).
This is one of the main instances where this article differs significantly from \cite{emv}, which relies on an effective closing lemma proven using significant constraints on the intermediate Lie algebra $\mathfrak{s}$, including that $\mathfrak{s}$ is semisimple and that there are finitely many such $\Fs$.
One reason for these constraints in \cite{emv} is the use of the (uniform) spectral gap for the action of the Lie group $S$ with Lie algebra $\mathfrak{s}$.
In the present article, we do \emph{not} rely on any properties of the intermediate Lie algebra $\mathfrak{s}$.

To outline our argument, we suppose for simplicity that $\G$ is absolutely almost simple and that $X$ is compact. 
Fix $T>0$ which is implicitly polynomial in $D$.
We cover $X$ with small boxes of the form 
$$y \exp({B_{\delta_1}^\Fs})\exp({B_{\delta_2}^\Fr})$$
where $\delta_1$ is much larger than $\delta_2$ and both are polynomial in $T$.
Consider the $U$-orbit for times $[0,T]$ through a point $x_0 \in xH$.
We may assume that $x_0u_t\exp(s)$ is generic for at least $90\%$ of $s \in B_{1}^\Fs$ and $t \in [0,T]$ in a sense similar to \eqref{eq:gen-pseudo};
this follows from an effective ergodic theorem (Proposition~\ref{prop:ptsalongU-orbitswith many close gen pts}) which was up to minor differences already established in \cite{emv}.
We need to show that $x_0u_t$ does not spend a disproportionate amount in any one of the above boxes.
Indeed, if we can avoid this exceptional situation, there will be two times $t_1,t_2$ so that $x_0u_{t_1},x_0u_{t_2}$ are at distance $\ll\delta_1$ and have displacement in the $\mathfrak{r}$-direction $\gg \delta_2$; realignment using $\exp({B_{1}^\Fs})$ then yields the desired points $x_1,x_2$.

The exceptional situation is addressed by an effective closing lemma in work of the author with Lindenstrauss, Margulis, Mohammadi, and Shah \cite{EffClosing}:
If $x_0u_t$ spends a disproportionate amount in one of the above boxes, we obtain information about the initial point $x_0$, namely that its orbit under $U$ stays close to a periodic orbit of some proper subgroup of $G$ for the time $T$.
However, effective avoidance results of Lindenstrauss, Margulis, Mohammadi, and Shah \cite{effectiveavoidance} imply that the set of such points has small measure in $xH$ (cf.~Corollary~\ref{cor:dio on H-orbit}).
In other words, we may choose the initial point appropriately and avoid the above problem.

\subsection{Overview of this article}
This article is structured as follows:
\begin{itemize}
\item In \S\ref{sec:prelim}, we set up notation and give some fundamental definitions (Sobolev norms, almost invariance and heights to name a few).
\item In \S\ref{sec:efferg}, we give a variant of the effective ergodic theorem proved in \cite[\S9]{emv}.
\item In \S\ref{sec:diopoints}, we recall the effective closing lemma from work with Lindenstrauss, Margulis, Mohammadi, and Shah \cite{EffClosing} and the effective avoidance result from \cite{effectiveavoidance}.
Simple corollaries for periodic $H$-orbits are derived.
\item In \S\ref{sec:addinv}, we show how pairs of nearby generic points give rise to additional almost invariance (as outlined in \S\ref{sec:intro-addinv}).
\item In \S\ref{sec:existence transverse}, we establish the existence of such generic points using \S\ref{sec:efferg},\S\ref{sec:diopoints}. 
\item In \S\ref{sec:proofmain}, we prove Theorem~\ref{thm:main} by an inductive procedure accumulating almost invariance.
\item In \S\ref{sec:exhaustive}, we give an extension of Theorem~\ref{thm:main} for suitable collections of subgroups of class $\classH$. In particular, we prove Theorem~\ref{thm:effequi orth}. 
\end{itemize}

\textbf{Acknowledgments:}
I would like to thank Elon Lindenstrauss and Amir Mohammadi for many discussions on this and related topics including, but not at all limited to, effective closing lemmas.
I am also grateful towards Manfred Einsiedler and Zhiren Wang for conversations about effective equidistribution results.
Lastly, I wish to thank the referees for their careful reading of this document as well as many useful comments.

\section{Preliminaries}\label{sec:prelim}

\subsection{Basic setup}\label{sec:basicsetup}
Let $\G$ be a connected semisimple $\Q$-group and fix a faithful representation $\G \hookrightarrow \SL_\dimemb$.
We may assume that the adjoint representation occurs as an irreducible subrepresentation.
To simplify notation, we identify $\G$ with its image.

The embedding $\G \hookrightarrow \SL_\dimemb$ fixes an integral structure on $\G$ which is inherited by its subgroups.
More concretely, we denote $\BM(\Z) = \BM(\Q)\cap \SL_N(\Z)$ for any $\Q$-subgroup $\BM < \G$.
Let $\Gamma < \G(\Z)$ be a congruence lattice, $G  = \G(\R)$ and $X = \lquot{\Gamma}{G}$.
The Lie algebra $\Fg$ of $\G$ also inherits an integral structure such that $\Fg(\Z) = \Fg(\Q) \cap \mathfrak{sl}_\dimemb(\Z)$. 
In particular, $\Fg(\Z)$ is $\Gamma$-invariant under the adjoint representation.

We fix a Euclidean norm $\norm{\cdot}$ on $\Mat_n(\R)$ with $\norm{[X,Y]}\leq \norm{X}\norm{Y}$ for all $X,Y$.
The Euclidean norm restricted to $\Fg$ induces a left-invariant metric $\metric(\cdot,\cdot)$ on $G$ which descends to a metric on $X$ also denoted by $\metric(\cdot,\cdot)$.
For $g \in \SL_N(\R)$ we write $\abs{g} = \min\{\norm{g}_\infty,\norm{g^{-1}}_\infty\}$ where $\norm{g}_\infty = \max_{i,j}|g_{ij}|$.
For an algebraic subgroup $\BM<\G$ the identity component is denoted by $\BM^\circ$ whereas for a (Hausdorff-) closed subgroup $M < G$ the identity component in the Hausdorff topology is $M^+$.

We fix a connected semisimple subgroup $H< G$ without compact factors throughout the whole article.
We may assume that $H$ is not contained in (the real points of) any proper normal $\Q$-subgroup of $\G$. 
Indeed, for any $g \in G$ and a proper normal subgroup $\BM \lhd \G$ the discriminant of the orbit $\Gamma \BM(\R)g$ depends only on $\BM$ (by the definition given in \S\ref{sec:discriminant} below).
Thus, Theorem~\ref{thm:main} is trivial for any periodic orbit $\Gamma g H \subset \Gamma \BM(\R)g$ simply by choosing a large enough implicit constant (depending only on $\G$).

Moreover, we fix a homomorphism $\SL_2(\R) \to H$ that projects non-trivially onto any simple factor of~$H$. We write $U< H$ for the image of
\begin{align*}
\Big\{
\begin{pmatrix}
1 & \ast \\ 0 & 1
\end{pmatrix}\Big\} < \SL_2(\R).
\end{align*}
We pick a unit vector $\ugen \in \Fg$ so that $U= \{u_t:t \in \R\}$ with $u_t = \exp(t\ugen)$.
The above assumptions imply that $U$ is not contained in any normal $\Q$-subgroup of $\G$.


As $H$ is semisimple, for any $H$-invariant subspace $V \subset \Fg$ there exists a $H$-invariant complement $W$ i.e.~a subspace with $V \oplus W = \Fg$.
The following is a simple upgrade of that statement.

\begin{lemma}
\label{lem:uniform invariant complement}
There exists a constant $\kappaundistcompl>0$ such that for every
$H$-invariant subspace $V\subset \Fg$ there exists an $H$-invariant
complement $W$ of $V$ satisfying
\begin{align*}
\norm{v+w}^2
\geq
\kappaundistcompl\big(\norm{v}^2+\norm{w}^2\big)
\end{align*}
for all $v\in V$ and $w\in W$.
\end{lemma}

We shall call such an $H$-invariant complement \emph{undistorted}.

\begin{proof}
Since $H$ is semisimple, we have $\Fg\simeq \bigoplus_i V_i\otimes_{D_i} M_i$ as $H$-representations, where the vector spaces $V_i$ are pairwise non-isomorphic irreducible $H$-representations, $D_i$ is the division algebra of $H$-equivariant endomorphisms of $V_i$, and $M_i$ is a $D_i$-module on which $H$ acts trivially.
In particular, $D_i$ is either $\R$, $\C$, or the algebra of Hamiltonian quaternions over $R$.
Under this identification, every $H$-invariant subspace is then of the form
$V=\bigoplus_i V_i\otimes_{D_i} N_i$ for (possibly trivial) submodules $N_i\subset M_i$.
Choosing complements $N_i'$ orthogonal to $N_i$ (with respect to a $D_i$-Hermitian inner product) and setting
$W=\bigoplus_i V_i\otimes_{D_i} N_i'$ gives an $H$-invariant complement, which is easily seen to satisfy the required property.
\end{proof}

Throughout this article, we shall choose undistorted $H$-invariant complements for Lie subalgebras $\Fs$ containing $\Fh$.
As opposed to \cite{emv}, no information is known on $\Fs$ and in particular, the complement cannot be chosen $\Fs$-invariant.

\subsubsection{Notations regarding implicit constants}
For two positive quantities $A,B$ (such as real valued positive functions) we write $A \ll B$ if there exists a constant $c>0$ such that $A \leq cB$. If we want to emphasize the dependency of the implicit $c$ on another quantity, say $m$, we write $A \ll_m B$.
Throughout this article, most implicit constants are allowed to depend on $\dimemb$ or $\dim(\G)$ and this dependency is usually omitted.
Similarly, we write $A^\star$ to mean $A^k$ for an absolute implicit constant $k>0$ which is allowed to depend on $\dimemb$.

\subsection{Height in the cusp and Sobolev norms}\label{sec:sob norms}
Define for any $x = \Gamma g \in X$ the height in the cusp
\begin{align*}
\height(x) = \sup\{ \norm{v}^{-1}: v \in \Ad(g^{-1})\Fg_\Z \text{ non-zero}\}.
\end{align*}
As mentioned earlier, this height is not to be confused with the notions of arithmetic heights defined e.g.~in \S\ref{sec:wedge rep}.
By Mahler's compactness criterion the height function $\height: X \to \R_{>0}$ is continuous and proper.
Let $X_{\eta}$ for any $\eta>0$ be the set of points of height at most $\eta^{-1}$. 
There is $\kappainjradius>0$ such that a uniform injectivity radius on $X_\eta$ is given by $\eta^{-\kappainjradius}$ for $\eta$ sufficiently small.
Moreover, by \cite[Lemma~2.7]{effectiveavoidance} there exist constants $\Adiam>0$ (depending on $\dimemb$) and $\Cdiam>0$ (depending on the geometry of $X$) such that for any $x \in X$ there exists a representative $g \in G$ with
\begin{align}\label{eq:diameter}
\abs{g}\leq \Cdiam \height(x)^{\Adiam}.
\end{align}
We refer to \cite[Thm.~1.7]{Mohammadi-Diameter} for a sharper understanding of the constant $C_1$ in terms of $G$ and $\Gamma$. 
Lastly, we remark that $\height(xg) \ll \abs{g} \height(x)$ for any $x \in X$, $g \in G$.

Fix an orthonormal basis of $\Fg$ and define for any $d \geq 0$ an $L^2$-Sobolev norm $\Sob_d$ of degree $d$ on $\compactsmooth(X)$ via
\begin{align*}
\Sob_d(f)^2 = \sum_{\mathcal{D}} \norm{(\height(\cdot)+1)^d \mathcal{D}f}^2
\end{align*}
where $\mathcal{D}$ runs over all monomials of degree $\leq d$ in the orthonormal basis of $\Fg$.
We record some properties of these Sobolev norms proven in \cite{emv}:
\begin{enumerate}[label=(S\arabic*)]
\item \label{item:sobineq}
There exists $\kappaSobineq \geq \dim(G)$ so that for all $f \in \compactsmooth(X)$ and $d \geq \kappaSobineq$
\begin{align*}
\norm{f}_\infty \ll \Sob_d(f).
\end{align*}
\item \label{item:soblip}
For any $d > \kappaSobineq$, $g \in G$ and $f \in \compactsmooth(X)$
\begin{align*}
\norm{f-g.f}_\infty \ll \metric(e,g)\Sob_d(f).
\end{align*}
\item \label{item:sobcont}
There exists $\ASobcont>0$ so that for any $d \geq 0$, $g \in G$, and $f \in \compactsmooth(X)$
\begin{align*}
\Sob_d(g.f) \ll_d \abs{g}^{\ASobcont d}\Sob_d(f).
\end{align*}
\item \label{item:sobprod}
For any $f_1,f_2 \in \compactsmooth(X)$
\begin{align*}
\Sob_d(f_1f_2) \ll_d \Sob_{d+\kappaSobineq}(f_1)\Sob_{d+\kappaSobineq}(f_2).
\end{align*}
\item \label{item:sobtrace}
For any $d$ there exists $d' > d$ such that
$\Tr(\Sob_{d'}|\Sob_d)< \infty$.
\end{enumerate}
We shall not use the relative trace estimate in \ref{item:sobtrace} explicitly, but mention here that it plays a crucial role in the proofs of the effective ergodic theorems in Propositions~\ref{prop:eff erg} and \ref{prop:ptsalongU-orbitswith many close gen pts}. 
We refer to \cite[\S5]{emv} for a thorough discussion influenced by work of Bernstein and Reznikov \cite{BernsteinReznikov}.

\subsection{Almost-invariant measures}

Let $\mu$ be a probability measure on $X$. Given $g \in G$ write $\mu^g$ for the pushforward of $\mu$ by $g$ i.e.~the measure given by $\mu^g(f) = \int f(xg) \de\mu(x)$ for $f \in \compactcont(X)$.
The following notion was introduced in \cite[\S3.10]{emv}.

\begin{definition}\label{def:almost invariant}
We say that $\mu$ is \emph{$\varepsilon$-almost invariant} (with respect to $\Sob_d$) under $g \in G$ if for all $f \in \compactsmooth(X)$
\begin{align*}
|\mu^g(f)-\mu(f)| \leq \varepsilon \Sob_d(f).
\end{align*}
Moreover, $\mu$ is $\varepsilon$-almost invariant under $Z \in \Fg$ if it is $\varepsilon$-almost invariant under $\exp(tZ)$ for $|t|\leq 2$.
It is $\varepsilon$-almost invariant under a Lie subalgebra $\Fs \subset \Fg$ if it is $\varepsilon$-almost invariant under all $Z \in \Fs$ with $\norm{Z} \leq 1$.
\end{definition}

When we write that a measure $\mu$ is $\ll \varepsilon$-almost invariant under $g\in G$ we mean that there is an absolute constant $C>0$ so that $\mu$ is $C\varepsilon$-almost invariant under $g$.
The same convention applies to elements of the Lie algebra.

Note that for $d \geq \kappaSobineq+1$ any measure $\mu$ is trivially $\ll \varepsilon$-almost invariant under any $g \in G$ with $\metric(g,e)\leq \varepsilon$ by \ref{item:soblip}, hence almost invariance is delicate for small elements of the group.
The following elementary property is verified in \cite[\S8.1]{emv} and allows one to pass from almost invariance under an element of $G$ to almost invariance under an element of the Lie algebra.
\begin{enumerate}[label=(AI)]
\item \label{item:alminv1}
If $\mu$ is $\varepsilon$-almost invariant under $\exp(Z)\in \Fg$ and $1 \leq c \leq 2\norm{Z}^{-1}$, then $\mu$ is $\ll (c\varepsilon+\norm{Z})$-almost invariant under $cZ$.
\end{enumerate}
Note that the analogous statement in an $S$-arithmetic setup (see e.g.~\cite{emmv}) fails completely as small elements of $\G(\Q_p)$ for a prime $p$ cannot be iterated to become of size $\approx 1$.

The effective generation results established in \cite[\S7]{emv} imply the following proposition.

\begin{proposition}[Effective generation]\label{prop:effective generation}
There exists a constant $\kappaeffgen>0$ depending on $H$ with the following property.
Let $\Fs \subset \Fg$ be a Lie subalgebra containing $\Fh$ and let $\Fr$ be an undistorted $H$-invariant complement to $\Fs$.
Let $\mu$ be an $H$-invariant probability measure on $X$.

Suppose that $\mu$ is $\varepsilon$-almost invariant (with respect to $\Sob_d$) under $\Fs$ and under a unit vector $Z \in \Fr$.
Then there exists a constant $c_1(d)$ and a Lie subalgebra $\Fs_\ast \supset \Fh$ with $\dim(\Fs_\ast)>\dim(\Fs)$ so that $\mu$ is $c_1(d)\varepsilon^{\kappaeffgen}$-almost invariant under $\Fs_{\ast}$.
\end{proposition}

\begin{proof}
Follows from the proof of \cite[Prop.~8.1]{emv}, see also footnote 21 therein.
\end{proof}

\begin{remark}\label{rem:notnormal}
We may and will assume that any proper intermediate Lie algebra $\Fh \subset \Fs \subsetneq \Fg$ is not normal in $\Fg$.
To that end, note that if $H$ has a periodic orbit in $X$ some conjugate $gHg^{-1}$ is defined over $\Q$ (or, more precisely, is equal to $\BH(\R)^+$ for some $\Q$-subgroup $\BH \in \G$).
Hence, the intersection of Lie ideals of $\Fg$ containing $\Fh$ is defined over $\Q$.
But we assumed that no proper Lie ideal defined over $\Q$ contains $\Fh$.
To summarize, if $H$ has a periodic orbit in $X$ there is no proper intermediate Lie ideal $\Fh \subset \Fs \subsetneq \Fg$ under our standing assumption on $H$.
\end{remark}

\subsection{Exterior representations and heights of subgroups}\label{sec:wedge rep}
For any $\Q$-subgroup $\BM < \G$ we define the $\Q$-vector space
\begin{align*}
\wedgespace{\BM} = \bigwedge^{\dim(\BM)} \mathfrak{g}.
\end{align*}
The integral structure on $\Fg$ induces an integral structure on $\wedgespace{\BM}$.
The Lie algebra $\Fm$ of $\BM$ defines the line $\bigwedge^{\dim(\BM)} \Fm$ in $\wedgespace{\BM}$ and we choose a primitive integral vector $\wedgevec{\BM}$ in that line (uniquely determined up to signs).
Lastly, we write $\wedgerep{\BM}$ for the (exterior) representation of $\G$ on $\wedgespace{\BM}$ induced by the adjoint representation and
\begin{align*}
\wedgeorbit{\BM}: \G \to \wedgespace{\BM},\ g \mapsto\wedgerep{\BM}(g^{-1}).\wedgevec{\BM}
\end{align*}
for the (right-)orbit map. When it is clear from context, we write simply $g.v = \wedgerep{\BM}(g)v$ for the action.
The \emph{height} of the subgroup $\BM$ is defined to be
\begin{align*}
\height(\BM) :=  \norm{\wedgevec{\BM}}
\end{align*}
where $\norm{\cdot}$ is the Euclidean norm on $\wedgespace{\BM}$ induced by the choice of Euclidean norm on $\Fg< \mathfrak{sl}_N(\R)$.

\subsubsection{Discriminant of an orbit}\label{sec:discriminant}

Recall that a connected $\Q$-subgroup $\BM<\SL_N$ is said to be of class $\classH$ if its radical is unipotent.
Note that a $\Q$-subgroup $\BM < \G$ is of class $\classH$ if and only if it has no nontrivial $\Q$-characters. 
By a theorem of Borel and Harish-Chandra \cite{BorelHarishChandra}, any orbit of the form $\Gamma M g$ for $M = \BM(\R)$ and $g \in G$ thus possesses a unique $g^{-1}Mg$-invariant probability measure. In particular, $\Gamma M g$ is closed.

For $\BM \in \classH$ a subgroup of $\G$ and $g \in G$ the \emph{discriminant} of the orbit $\Gamma \BM(\R) g$ is
\begin{align*}
\disc(\Gamma \BM(\R) g) := \norm{\wedgeorbit{\BM}(g)}.
\end{align*}
Note that the discriminant depends only on the orbit as $\wedgeorbit{\gamma\BM\gamma^{-1}}(\gamma g)= \wedgeorbit{\BM}(g)$ for any $\gamma\in \Gamma$.

\begin{remark}
The discriminant $\disc(\Gamma \BM(\R) g)$ for $\BM$ semisimple is comparable to the volume $\vol(\Gamma \BM(\R) g)$.
We shall not need such a comparison in the current article and hence omit a more precise statement -- see \cite[\S17]{emv} for a discussion.
\end{remark}

We also define
\begin{align}\label{eq:ethdef}
\minobs(xH) = \minobs_{\classH}(xH) = \inf\{\disc(\Gamma \BM(\R) g): \G \neq \mathbf{M} \in \classH \text{ with } xH \subset \Gamma \BM(\R) g\}
\end{align}
for later convenience.



%
%
%
%
%

\section{Effective ergodic theorem}\label{sec:efferg}
Let $\mu$ be the $H$-invariant probability measure on a periodic $H$-orbit in $X$.
In this section, we use the uniform spectral gap of the $H$-representation on $L^2_0(\mu)$ to exhibit an abundance of `generic' points (a version of an effective ergodic theorem).
While spectral gap is also used in the final step of the proof (Proposition~\ref{prop:alminv under G} below), the application of spectral gap is deepest here.
This is in contrast to \cite{emv} where the gap was also used for an effective closing lemma for intermediate groups (over which the current article has no control).

By Clozel's property $(\tau)$, there is a constant $p_G>1$ depending only on $\dim(\G)$ such that the $H$-represen\-tation on $L^2_0(\mu)$ is $\frac{1}{p_G-1}$-tempered\footnote{Here, a unitary representation $\pi$ of $H$ is $\frac{1}{m}$-tempered for $m \in \mathbb{N}$ if $\pi^{\otimes m}$ is tempered or equivalently if there is a dense set of vectors $\mathcal{V}$ such that for all $v,w \in \mathcal{V}$ the matrix coefficient $H \ni h \mapsto \langle \pi(h)v,w\rangle$ is in $L^{2m+\varepsilon}$ for all $\varepsilon>0$.} --- see e.g.~\cite[Lemma 6.6]{emv}. We set $\tempered= 20(p_G+1)$.

We define an effective notion of Birkhoff genericity following \cite{emv}.

\begin{definition}\label{def:discrepancy}
For any $f \in C_c(X)$ and $n>0$ the \emph{discrepancy} $D_n(f)$ of $f$ (with respect to $U$ and $\mu$) is given by
\begin{align*}
D_n(f) = \frac{1}{(n+1)^{\tempered}-n^{\tempered}} \int_{n^{\tempered}}^{(n+1)^{\tempered}} f(xu_t)\de t - \int f \de \mu.
\end{align*}
Given positive numbers $k_1 < k_2$ a point $x \in X$ is said to be \emph{$[k_1,k_2]$-generic} (for $U$ and $\mu$) with respect to $\Sob_d$ if 
\begin{align*}
|D_n(f)(x)| \leq \tfrac{1}{n}\Sob_d(f)
\end{align*}
for all $n \in [k_1,k_2] \cap \Z$ and all $f \in C_c^\infty(X)$.
We say that $x$ is \emph{$k_0$-generic} if it is $[k_0,k_1]$-generic for all $k_1 > k_0$.
\end{definition}

The following ergodic theorem is established in {\cite[Prop.~9.1]{emv}}.

\begin{proposition}\label{prop:eff erg}
There is $d_1\geq 1$ (depending only on $G,H,U$ and not on $\mu$) so that for any $d \geq d_1$ the $\mu$-measure of the set of points which are not $k_0$-generic with respect to $\Sob_d$ is $\ll_d k_0^{-1}$.
\end{proposition}

We remark that the proof of Proposition~\ref{prop:eff erg} does \emph{not} use the centralizer assumption while its refinement for almost invariant measures \cite[Prop.~9.2]{emv} does.
Indeed, this extension requires an effective generation result for large balls in intermediate Lie subgroups $H < S < G$ (see also \cite[Lemma~8.2]{emv}).
In the context of \cite{emv} these intermediate groups are automatically semisimple as opposed to the current article.

Here, we use an adapted version which does not rely on the centralizer assumption, but follows by the same methods.

\begin{proposition}\label{prop:ptsalongU-orbitswith many close gen pts}
Let $d \geq \kappaSobineq + 1$.
There exist $\beta=\beta(d)>0$, $d' > d$ and $\varepsilon_0>0$ with the following property.
Suppose that $\mu$ is $\varepsilon$-almost invariant under a Lie subalgebra $\Fs \supset \Fh$ with respect to $\Sob_d$ for some $\varepsilon \leq \varepsilon_0$.

Let $\Omega \subset \Fs$ be a compact neighborhood of $0$ contained in the unit ball.
For any $t_0 >0$ and $k_0>0$, the fraction of points
\begin{align*}
(x,t,Z) \in X \times [0,t_0] \times \Omega 
\end{align*}
for which $xu_t\exp(Z)$ is not $[k_0,\varepsilon^{-\beta}]$-generic with respect to $\Sob_{d'}$ is $\ll_d k_0^{-1}$.
Here, we use the Lebesgue measure on $\Omega$.
\end{proposition}

\begin{proof}[Proof of Proposition~\ref{prop:ptsalongU-orbitswith many close gen pts}]
The proof is largely analogous to the proof of \cite[Prop.~9.2]{emv}; we give it for completeness.
Invariance of $\mu$ under $U$ implies for $n \in \N$ 
\begin{align*}
\frac{1}{t_0\vol(\Omega)} &\int_{X\times [0,t_0] \times \Omega} |D_n(f)(xu_t\exp(Z))|^2 \de \mu(x) \de t\de Z \\
&= \frac{1}{\vol(\Omega)} \int_{X \times \Omega} |D_n(f)(x\exp(Z))|^2 \de \mu(x) \de Z\\
&=  \frac{1}{\vol(\Omega)} \int_{X \times \Omega} |D_n(f)(x)|^2 \de \mu^{\exp(Z)}(x) \de Z.
\end{align*}
By $\varepsilon$-almost invariance $\mu$ under $\Fs$ (w.r.t.~$\Sob_d$) we have that $|\mu^{\exp(Z)}(F)-\mu(F)|\leq \varepsilon \Sob_d(F)$ for all $F \in C_c^\infty(X)$ and $Z \in \Omega$.
This implies for any $Z \in \Omega$
\begin{align*}
\int_{X} |D_n(f)(x)|^2 \de \mu^{\exp(Z)}(x)
&\leq \int_{X} |D_n(f)(x)|^2 \de \mu(x) + \varepsilon \Sob_d\big(|D_n(f)|^2\big).
\end{align*}
By \ref{item:sobprod} and \ref{item:sobcont}
\begin{align*}
\Sob_d\big(|D_n(f)|^2\big)
\ll_d  \Sob_{d+\kappaSobineq}(D_n(f))^2
\ll_d n^\star \Sob_{d+\kappaSobineq}(f)^2.
\end{align*}
We obtain  $\varepsilon \Sob_d\big(|D_n(f)|^2\big) \ll n^{-4}\Sob_{d+\kappaSobineq}(f)^2$ for an appropriate choice of $\beta$ and $n \leq \varepsilon^{-\beta}$.
From here, one can conclude the proof as in \cite[Prop.~9.1]{emv} (increasing $d$).
\end{proof}

\section{Effective avoidance results}\label{sec:diopoints}

\subsection{Diophantine points}
We recall the notion of Diophantine points introduced by Lindenstrauss, Margulis, Mohammadi, and Shah in \cite{effectiveavoidance} in slightly adapted notation.
Fix a decreasing function
\begin{align*}
\dioeps: (0,\infty)\to (0,1).
\end{align*}
Recalled that we chose a unit vector $\ugen \in \Fg$ so that $U =\{ u_t =\exp(t\ugen):t\in \R\}$.

\begin{definition}[Diophantine points]\label{def:diophantine}
A point $x = \Gamma g\in X$ is $(\dioeps,T)$-Diophantine if for any proper non-trivial subgroup $\BM < \G$ of class $\classH$ we have
\begin{align*}
\norm{\ugen \wedge \wedgeorbit{\BM}(g)}\geq \dioeps(\norm{\wedgeorbit{\BM}(g)})
\end{align*}
whenever $\norm{\wedgeorbit{\BM}(g)} < T$.
\end{definition}

For convenience of the reader and to introduce constants we recall here results proven in \cite{effectiveavoidance}.
We include the clarifying dependencies of the constants while noting that they carry little importance within this article.

\begin{theorem}[{\cite[Theorem~3.2]{effectiveavoidance}}]\label{thm:effectiveavoidance}
There exist constants $\AeffavT\geq 1$ and $\kappaeffav\in (0,1)$ depending on $\dimemb$, $\Ceffav>0$ depending on $\dimemb$ and $\height(\G)$, and $\Ceffavgeom>0$ depending on $\dimemb$, $\height(\G)$ and $\Gamma$ with the following property.
Let $g \in G$, let $T\geq 0$, let $\eta \in (0,\frac{1}{2})$, and let $t_1 < t_2$. 
Suppose that for all $s>0$
\begin{align}\label{eq:dioepsassum}
\dioeps(s) \leq \tfrac{1}{\Ceffavgeom}s^{-\AeffavT}\eta^{\AeffavT}.
\end{align}
Then at least one of the following options hold:
\begin{enumerate}[(I)]
\item\label{item:effav1}
We have
\begin{align*}
\big|\{t\in [t_1,t_2]: \Gamma g u_t \not\in X_\eta \text{ or } \Gamma g u_t \text{ not } (\dioeps,T)\text{-Diophantine} \}  \big|
< \Ceffavgeom \eta^{\kappaeffav} |t_2-t_1|.
\end{align*}
\item\label{item:effav2}
There exists a nontrivial proper subgroup $\BM < \G$ of class $\classH$ so that for all $t \in [t_1,t_2]$
\begin{align*}
\norm{\wedgeorbit{\BM}(gu_t)} &\leq (\Ceffav \abs{g}^{\AeffavT}+ \Ceffavgeom T^{\AeffavT}) \eta^{-\AeffavT},\\
\norm{\ugen \wedge \wedgeorbit{\BM}(gu_t)}
&\leq |t_2-t_1|^{-\kappaeffav} (\Ceffav \abs{g}^{\AeffavT}+ \Ceffavgeom T^{\AeffavT}) \eta^{-\AeffavT}.
\end{align*}
\item\label{item:effav3} 
There exists a nontrivial proper normal subgroup $\BM < \G$ with
\begin{align*}
\norm{\ugen \wedge \wedgevec{\BM}}
\leq \dioeps\Big(\height(\BM)^{1/\AeffavT}\eta/\Ceffavgeom\Big)^{1/\AeffavT}.
\end{align*}
\end{enumerate}
\end{theorem}

We also record the following simple corollary to Theorem~\ref{thm:effectiveavoidance}.
As in the introduction, let $\minht(xH)$ be the smallest height in the cusp attained on a periodic $H$-orbit $xH$.

\begin{corollary}[Diophantine points on $H$-orbits]\label{cor:dio on H-orbit}
There exist $\AeffavcorH\geq 1$ depending on $\dimemb$ and $\CeffavcorH >1$ depending on $\dimemb$, $\height(\G)$, $\Gamma$ and $U$ with the following property.
Let $\mu=\mu_{xH}$ be the Haar probability measure on a periodic orbit $xH= \Gamma g H$.
Let $T \geq 1$, $\eta \in (0,\frac{1}{2})$ and suppose that
\begin{align}\label{eq:dioepsest Hcase}
\dioeps(s)\leq \tfrac{1}{\CeffavcorH} s^{-\AeffavcorH} \eta^{\AeffavcorH}
\end{align}
Then at least one of the following two options holds.
\begin{enumerate}
\item We have
\begin{align*}
\mu\big( \{y\in X: y \text{ is not } (\dioeps,T)\text{-Diophantine or } y \not\in X_{\eta}\}\big)
\leq  \Ceffavgeom \eta^{\kappaeffav}.
\end{align*}
\item There exists a proper subgroup $\mathbf{M} \in \classH$ of $\G$ such that the orbit $\Gamma \BM(\R) g$ contains $xH$ and satisfies
\begin{align*}
\disc(\Gamma \BM(\R) g) \leq \CeffavcorH\minht(xH)^\Adiocond T^\Adiocond\eta^{-\Adiocond}.
\end{align*}
Equivalently, in view of \eqref{eq:ethdef}, 
\begin{align*}
\minobs(xH) \leq \CeffavcorH\minht(xH)^\Adiocond T^\Adiocond\eta^{-\Adiocond}.
\end{align*}
\end{enumerate}
\end{corollary}

\begin{proof}
We assume that $\dioeps(\cdot)$ satisfies \eqref{eq:dioepsest Hcase} for some sufficiently large $\CeffavcorH\geq \Ceffavgeom$ to be determined in the proof.
We wish to use \cite[Theorem~3.2]{effectiveavoidance} (Theorem~\ref{thm:effectiveavoidance}) for a generic point on the $H$-orbit $xH$; we may assume without loss of generality that $x = \Gamma g$ itself is Birkhoff generic with respect to $U$ and the invariant probability measure $\mu = \mu_{xH}$.
Similarly, we may suppose that $|g| \leq 2 \Cdiam \minht(xH)^{\Adiam}$ by \eqref{eq:diameter}.

Consider Theorem~\ref{thm:effectiveavoidance} for intervals $[0,t']$, $t' \in \N$, and the point $x$. Then one of the three Options \ref{item:effav1}--\ref{item:effav3} occurs infinitely often and by changing to a subsequence $t_\ell$ of the integers we assume that it always occurs.
If \ref{item:effav1} holds for all intervals $[0,t_\ell]$, we may pass to the limit 
as $\ell \to \infty$ to obtain $(1)$ (recall that $x$ is Birkhoff generic).

Suppose now that for any $\ell$ there exists a nontrivial non-normal subgroup $\mathbf{M}_\ell \in \classH$ as in \ref{item:effav2}.
In particular, $\norm{\wedgeorbit{\BM_\ell}(g)} \leq (\Ceffav \abs{g}^{\AeffavT}+ \Ceffavgeom T^{\AeffavT}) \eta^{-\AeffavT}$ for all $\ell$.
As there are finitely many $\Q$-groups satisfying this bound, we may pass to a subsequence and assume $\BM = \BM_\ell$ for all $\ell$.
We have for all $t \in [0,t_\ell]$
\begin{align}
\norm{\wedgeorbit{\BM}(gu_t)} &\leq (\Ceffav \abs{g}^{\AeffavT}+ \Ceffavgeom T^{\AeffavT}) \eta^{-\AeffavT}\label{eq:rulingoutnon-normal1}\\
\norm{\ugen \wedge \wedgeorbit{\BM}(gu_t)}
&\leq t_\ell^{-\kappaeffav} (\Ceffav \abs{g}^{\AeffavT}+ \Ceffavgeom T^{\AeffavT}) \eta^{-\AeffavT}\label{eq:rulingoutnon-normal2}
\end{align}
for all $t \in [0,t_\ell]$.
Letting $\ell \to \infty$ we obtain $\norm{\ugen\wedge \eta_{\BM}(g)}=0$ from \eqref{eq:rulingoutnon-normal2} or equivalently $U \subset g^{-1}\BM(\R) g$.
This shows that $xU \subset \Gamma \BM(\R) g$ and hence $xH  \subset \Gamma \BM(\R) g$ by taking the closure.
In particular, $H \subset g^{-1}\BM(\R) g$ and the second case of the proposition holds with the group $\BM$ by \eqref{eq:rulingoutnon-normal1}.

Suppose Option \ref{item:effav3} in Theorem~\ref{thm:effectiveavoidance} holds. In particular, there is a proper normal subgroup $\mathbf{M} = \mathbf{M}_\ell$ with
\begin{align}\label{eq:rulingoutnormal}
\norm{\ugen\wedge \wedgevec{\BM}} \leq \dioeps\big(\height(\mathbf{M})^{\frac{1}{\AeffavT}}\eta/\Ceffavgeom\big)^{\frac{1}{\AeffavT}}.
\end{align}
Notice that $\ugen\wedge \wedgevec{\BM'} \neq 0$ for any proper normal subgroup $\BM' \lhd \G$.
Indeed, if  $\ugen\wedge \wedgevec{\BM'} = 0$ then $U$ and all its conjugates by elements of $H$ are contained in $\BM'(\R)$. But the conjugates of $U$ generate $H$ and so $H \subset \BM'(\R)$ contradicting our assumption in \S\ref{sec:basicsetup}.
Thus, choosing $\CeffavcorH$ large enough with $C_4^{-1}\ll \min_{\G \neq \BM'\lhd \G}\norm{\ugen\wedge \wedgevec{\BM}}^\star$ we contradict \eqref{eq:rulingoutnormal}.
Here, note that the minimum is well-defined and positive as $\G$ has finitely many normal subgroups. 
This concludes the proof.
\end{proof}

In the remainder of this article, we will consider functions $\dioeps(\cdot)=\dioeps_\eta(\cdot)$ of the form
\begin{align}\label{eq:dioepschoice}
\dioeps_\eta(s) =  \tfrac{1}{\CeffavcorH} s^{-\AeffavcorH} \eta^{\AeffavcorH}
\end{align}
for $\eta \in (0,\frac{1}{2})$.
Here, $\eta$ will be fixed later.
Thus, if
\begin{align*}
T < \CeffavcorH^{-\frac{1}{\Adiocond}}\frac{\minobs(xH)^{\frac{1}{\Adiocond}}\eta}{\minht(xH)}
\end{align*}
the existence of many $(\psi_\eta,T)$-Diophantine points on the $H$-orbit $xH$ is guaranteed by Corollary~\ref{cor:dio on H-orbit}.

\subsection{Effective closing lemma}

This paper relies heavily on the following effective closing lemma due to Lindenstrauss, Margulis, Mohammadi, Shah, and the author from \cite{EffClosing}.
For a subalgebra $\Fs < \Fg$ we write $\hat{v}_{\Fs}$ for the point $\Fs$ defines on the projective space $\mathbb{P}^{\dim(\Fs)}(\Fg)$.
We equip $\mathbb{P}^{\dim(\Fs)}(\Fg)$ with the metric $\metric(\cdot,\cdot)$ (the Fubini-Study metric) given by $\metric(\hat{v},\hat{w}) = \min\{\norm{v+w},\norm{v-w}\}$ for unit vectors $v \in \hat{v},w\in \hat{w}$.

\begin{theorem}[{\cite[Theorem~1.1]{EffClosing}}]\label{thm:effectiveclosing}
There exist constants $\Aeffclos,\Aeffclosres>1$ depending only on $N$, and $E>0$ depending on $N,G,\Gamma$ with the following property.
Let $\Fs < \Fg$ be a Lie subalgebra which is not an ideal. 
Assume $\tau>0$, $T>R>E\tau^{-\Aeffclos}$.
Let $x = \Gamma g \in X_{\tau}$ be a point.

Suppose that there exists a measurable subset $\mathcal{E} \subset [-T,T]$ with the following properties:
\begin{enumerate}[(a)]
\item $|\mathcal{E}|>TR^{-1/\Aeffclos}$.
\item For any $s,t\in \mathcal{E}$ there exists $\gamma_{s,t} \in \Gamma$ with 
\begin{gather*}
\norm{u_{-s}g^{-1}\gamma_{s,t}gu_t}\leq R^{1/\Aeffclos}\\
\metric(u_{-s}g^{-1}\gamma_{s,t}gu_t. \hat{v}_{\Fs}, \hat{v}_{\Fs})\leq R^{-1}.
\end{gather*}
\end{enumerate}
Then one of the following is true:
\begin{enumerate}
\item There exist a nontrivial proper subgroup $\BM\in\classH$ 
so that the following hold for all $t \in [-T,T]$:
\begin{align*}
\norm{\eta_{\BM}(gu_t)}&\leq R^{\Aeffclosres},\\
\norm{\ugen\wedge{\eta_{\BM}(gu_t)}}&\leq T^{-1/\Aeffclosres}R^{\Aeffclosres}.
\end{align*}
\item There exist a nontrivial proper normal subgroup $\BM\in\classH$ with
\begin{align*}
\norm{\ugen\wedge \wedgevec{\BM}}&\leq R^{-1/\Aeffclosres}.
\end{align*}
\end{enumerate}    
\end{theorem}

As it is an essential input to the present article, we give a vague overview of the ideas that go into the proof of Theorem~\ref{thm:effectiveclosing}, omitting some (relevant) details.
One analyses the Zariski closure $\mathbf{M}_s$ of the group $\gamma_{s,t}$ where $s\in \mathcal{E}$ is a fixed point with `good density' and $t\in \mathcal{E}$ varies (in practice not over all of $[-T,T]$).
This good density hypothesis allows one to apply a stabilization argument to ensure good control on the height of $\mathbf{M}_s$.
The assumption in (b) shows that the Lie algebra $\mathfrak{s}$ is in a suitable sense almost normalized by the generators $\gamma_{s,t}$ of the group $\mathbf{M}_s$.
An argument based on an effective variant of the Lojasiewicz inequality shows that a small perturbation of $\mathfrak{s}$ is normalized by $\mathbf{M}_s(\R)$.
In particular, $\mathbf{M}_s \neq \G$.
From this, one can show that $xu_s$ is not Diophantine (with respect to a suitable function) and so the main theorem of \cite{effectiveavoidance} applies to show the conclusions.

We will use the effective closing lemma only in the weaker form of the following corollary.

\begin{corollary}\label{cor:effective closing}
There exists $\Aeffcloscor>0$ depending only on $\dimemb$ and $\Ceffcloscor>0$ depending on $N,G,\Gamma$ with the following property.
Let $\Fs< \Fg$ be a subalgebra.
Let $\eta\in (0,1/2)$ and $T \geq \Ceffcloscor\eta^{-\Aeffcloscor}$. 
Let $x \in X_{\eta}$ be $(\psi_\eta,T)$-Diophantine.
Suppose that $\mathcal{E}\subset [0,T]$ is a measurable subset satisfying the following:
\begin{enumerate}[(a)]
\item $|\mathcal{E}| > T^{1-1/\Aeffcloscor}$
\item For any $s,t \in \mathcal{E}$ there exist $g_{s,t} \in G$ with $|g_{s,t}| \leq 2$,
\begin{align}\label{eq:closing-displ}
xu_{s} = xu_{t}g_{s,t},
\end{align}
and
\begin{align}\label{eq:closing-bounddispl}
\metric(g_{s,t}.\hat{v}_{\Fs}, \hat{v}_{\Fs}) \leq T^{-1}.
\end{align}
\end{enumerate}
Then $\Fs$ is a Lie ideal.
\end{corollary}

\begin{proof}[Proof of Corollary~\ref{cor:effective closing}]
We apply Theorem~\ref{thm:effectiveclosing} with $\tau = \eta$ and $R = T^\theta$ for some small $\theta$ to be determined. In view of our assumptions, we may assume $R$ is bigger than any given fixed constant below.
We assert that (a) in Theorem~\ref{thm:effectiveclosing} is satisfied by choosing $\Aeffcloscor> \Aeffclos/ \theta$.
Moreover, writing $x = \Gamma g$ and using the assumptions, there is for any $s,t \in \mathcal{E}$ a lattice element $\gamma_{s,t} \in \Gamma$ with $gu_s = \gamma_{s,t}gu_t g_{s,t}$.
In particular, $u_{-s}g^{-1}\gamma_{s,t}gu_t = g_{s,t}^{-1}$.
Thus, $|g_{s,t}| \leq 2$ implies the first bound in (b) of Theorem~\ref{thm:effectiveclosing} and \eqref{eq:closing-bounddispl} yields the second bound.

By contradiction, assume that $\Fs$ is not an ideal.
Suppose conclusion (1) in Theorem~\ref{thm:effectiveclosing} holds for some $\BM \in \mathcal{H}$.
In particular,
\begin{align*}
T^{-\star}R^{\star} \geq \norm{\ugen \wedge \eta_{\BM}(g)}
\geq \psi_\eta(\norm{\eta_{\BM}(g)}) \gg \eta^\star R^{-\star}
\end{align*}
and so $T \ll \eta^{\star}R^{\star}$ which is a contradiction in view of our assumptions and for $\theta$ sufficiently small.

Suppose conclusion (2) in Theorem~\ref{thm:effectiveclosing} holds for some $\BM \lhd \G$.
Similarly to the previous case
\begin{align*}
R^{-\star} \gg \norm{\ugen \wedge \wedgevec{\BM}} \geq \psi_\eta(\height(\BM)) \gg \eta^\star \big(\max_{\BM \lhd \G} \height(\BM) \big)^{-\star} \gg \eta^{\star}
\end{align*}
where we used that $\G$ has finitely many normal subgroups (see Lemma~\ref{lem:normalsubgroups} for a more explicit estimate).
This proves the corollary.
\end{proof}

%
%

\section{Additional almost invariance}\label{sec:addinv}

In the following, we use polynomial divergence of $U$-orbits to show that transversal generic points yield additional almost invariance (see \S\ref{sec:intro-addinv} for an introductory discussion).
The following proposition is a refinement of \cite[Prop.~10.1]{emv} to include almost centralized displacements.

\begin{proposition}\label{prop:add inv}
Let $k_0,d \geq 1$ and $\varepsilon>0$.
Let $\mu$ be the Haar probability measure on a periodic $H$-orbit and assume that $\mu$ is $\varepsilon$-almost invariant under an intermediate Lie algebra $\Fh \subset \Fs \subset \Fg$ (with respect to $\Sob_d$).
Let $\Fr \subset \Fg$ be an undistorted $H$-invariant complement to $\Fs$.
Suppose that there exist two points $x_1,x_2\in X$ satisfying that
\begin{itemize}
\item $x_2 = x_1 \exp(r)$ for some $r \in \Fr$ with
\begin{align}\label{eq:addinv-displcond}
\varepsilon^{1/4} \leq \norm{r} \leq 1,
\end{align}
\item and $x_1,x_2$ are $[k_0,\varepsilon^{-1}]$-generic (with respect to $\Sob_d$).
\end{itemize}
Then 
$\mu$ is $\ll_d k_0^\star \norm{r}^{\star}$-almost invariant (with respect to $\Sob_d$) under some $Z \in \Fr$ with $\norm{Z} = 1$.
\end{proposition}

We require the following simple lemma which is a consequence of $\Fr$ being undistorted. 

\begin{lemma}\label{lem:EMV10.1}
There is $\kappacomm\in (0,1)$ (not depending on $\Fs$ or $\Fr$) so that the following is true:
for any $r_1,r_2 \in \mathfrak{r}$ with $\norm{r_1},\norm{r_2} \leq \kappacomm$ one can write
\begin{align}\label{eq:emv eq 10.1}
\exp(r_1)\exp(r_2)^{-1} = \exp(Z_{\Fr})\exp(Z_{\Fs})
\end{align}
for some $Z_{\Fr} \in \mathfrak{r}$ and $Z_{\Fs} \in \Fs$ so that 
\begin{align*}
\norm{Z_{\Fs}} \leq \norm{r_1-r_2},\
\kappacomm \norm{r_1-r_2} \leq \norm{Z_{\Fr}} \leq \tfrac{1}{\kappacomm}\norm{r_1-r_2}.
\end{align*}
\end{lemma}

Note that a similar statement is contained in \cite[(10.1)]{emv}.

\begin{proof}
The derivative of the map $\Phi:(r,s)\in \Fr\times \Fs \mapsto \log(\exp(r)\exp(s))$ at the origin is the identity and, 
as can be seen from the Baker-Campbell-Hausdorff formula,
$\norm{\mathrm{D}_u \Phi -\mathrm{id}} \ll \norm{u}$.
Here, the implicit constant is independent of $\mathfrak{s}$ and $\mathfrak{r}$ because $\Fr$ is undistorted.
By the quantitative inverse function theorem, this uniform estimate
allows us to choose a ball around the origin, independent of $\Fs$ and
$\Fr$, on which $\Phi$ is invertible with uniformly bounded inverse
derivative. Shrinking $\kappacomm$ if necessary, this gives
\eqref{eq:emv eq 10.1} and the stated bounds.
\end{proof}

We also require the following simple rescaling property of polynomials.

\begin{lemma}\label{lem:polynomial on shorter interval}
Let $p\in\R[t]$ be a polynomial of degree at most $d$ with $p(0)=0$.
Then we have for all $0<t\leq T$
\begin{align*}
\sup_{s\in[0,t]} |p(s)|
\ll_d
\frac{t}{T}
\sup_{s\in[T/2,T]} |p(s)|.
\end{align*}
\end{lemma}

\begin{proof}
Replacing $p(\cdot)$ with $p(T\cdot)$, we may assume $T=1$.
Write $p(s)=s q(s)$ for some $q\in \R[t]$ with degree at most $d-1$.
By equivalence of norms
\begin{align*}
\sup_{s\in[0,1]} |q(s)|
\ll_d
\sup_{s\in[1/2,1]} |q(s)|
\leq 2
\sup_{s\in[1/2,1]} |p(s)|.
\end{align*}
By definition of $q$ we have
\begin{align*}
\sup_{s\in[0,t]} |p(s)|
\leq
t\sup_{s\in[0,1]} |q(s)|
\end{align*}
and the claim follows.
\end{proof}

\begin{proof}[Proof of Proposition~\ref{prop:add inv}]
In view of the conclusion, we may and will assume $k_0 \ll \norm{r}^{-1/A}$ for some absolute large constant $A>0$.
Indeed, $\mu$ is trivially $\ll 1$-almost invariant under any vector in $\Fg$ of norm $1$ by \ref{item:sobineq}.
To simplify notation, we set $t_0 = k_0^{\tempered}$ following the definition of the discrepancy operator in Definition~\ref{def:discrepancy} (recall $\tempered= 20(p_G+1)$).
We prove the proposition by case distinction according to whether or not
\begin{align}\label{eq:casedistinction addinvproof}
\sup_{t \in [t_0,\varepsilon^{-\tempered}]} \norm{\Ad_{u_{-t}}(r)-r} \leq 1.
\end{align}
Note that $p(t) = \Ad_{u_{-t}}(r)-r$ is a polynomial in $t$ with $p(0) = 0$; there is no control on the speed of growth of $p(\cdot)$ (from below) e.g.~since $r$ could be $U$-invariant.

Suppose first that \eqref{eq:casedistinction addinvproof} holds.
Since $p(0) = 0$ and $t_0 \leq \varepsilon^{-K}/2$, Lemma~\ref{lem:polynomial on shorter interval} applied to coordinates of $p(\cdot)$ in an orthonormal basis implies that 
\begin{align}\label{eq:poly on shorter interval}
\sup_{t \in [t_0,t_1]} \norm{p(t)} 
\ll t_1 \varepsilon^{\tempered}\qquad \text{ for all }t_1 \leq \varepsilon^{-\tempered}.
\end{align}
The implicit constant in \eqref{eq:poly on shorter interval} depends only on the degree.

By the genericity assumption, we have for any $n \in [k_0,\varepsilon^{-1}]$, $f \in \compactsmooth(X)$ and $i=1,2$
\begin{align*}
\left|\mu(f) -  \frac{1}{(n+1)^{\tempered}-n^{\tempered}} \int_{n^{\tempered}}^{(n+1)^{\tempered}} f(x_iu_t)\de t\right|
\leq \frac{1}{n}\Sob_d(f).
\end{align*}
Also, note that
\begin{align*}
f(x_2u_t) &= f(x_1\exp(r)u_t) = f\big(x_1u_t\exp(\Ad_{u_{-t}}(r)\big) \\
&=f\big(x_1u_t\exp(r + p(t))\big).
\end{align*}
Notice that
\begin{align*}
\exp(r+p(t)) = \exp(r)(I+O(\norm{p(t)}))
\end{align*}
so that for all $t \in [0,2\varepsilon^{-\tempered/2}]$ using \ref{item:soblip}
\begin{align*}
f(x_2u_t) &= f(x_1u_t \exp(r+p(t)) = f(x_1u_t\exp(r)) + O(\varepsilon^{\tempered/2} \Sob_d(f)).
\end{align*}
Thus, we have for $n \in [k_0,\varepsilon^{-1/2}]$
\begin{align*}
\mu(f) &= \frac{1}{(n+1)^{\tempered}-n^{\tempered}} \int_{n^{\tempered}}^{(n+1)^{\tempered}} f(x_2u_t)\de t + O(\tfrac{1}{n}\Sob_d(f)) \\
&=\frac{1}{(n+1)^{\tempered}-n^{\tempered}} \int_{n^{\tempered}}^{(n+1)^{\tempered}} f(x_1u_t\exp(r))\de t + O(\tfrac{1}{n}\Sob_d(f) + \varepsilon^{\tempered/2}\Sob_d(f))\\
&=\mu(\exp(r).f) + O(\tfrac{1}{n}\Sob_d(f) + \varepsilon^{\tempered/2}\Sob_d(f))
\end{align*}
where we used $\Sob_d(\exp(r).f) \ll \Sob_d(f)$ (cf.~\ref{item:sobcont}) in the last step.
We now choose $n = \lfloor \varepsilon^{-1/2} \rfloor$ 
and obtain that $\mu$ is $\ll \varepsilon^{1/2}$-almost invariant under $\exp(r)$.
Letting $Z = \frac{r}{\norm{r}}$ we deduce from \ref{item:alminv1} that $\mu$ is $\ll(\frac{1}{\norm{r}}\varepsilon^{1/2} + \norm{r})$-almost invariant under $Z$. 
If \eqref{eq:addinv-displcond} holds, $\mu$ is $\ll \norm{r}$-almost invariant under $Z$.
This proves the proposition assuming that \eqref{eq:casedistinction addinvproof} holds.

We suppose now that \eqref{eq:casedistinction addinvproof} fails. 
In this case, one can proceed exactly as in \cite[Proposition~10.1]{emv}. Define
\begin{align*}
T = \tfrac{1}{2}\inf\big\{t \in [t_0,\varepsilon^{-\tempered}]: \norm{p(t)} \geq \kappacomm\big\}
\end{align*}
where $\kappacomm>0$ is as in \eqref{eq:emv eq 10.1}.
Here, note that $T$ is well-defined as \eqref{eq:casedistinction addinvproof} is assumed to fail.
Since the coefficients of $p$ are bounded in terms of $\norm{r}$, we have $T \gg \norm{r}^{-\star}$.
The polynomial map $q(s) = p(sT)$  satisfies $\sup_{s\in [0,2]}\norm{q(s)} = \kappacomm$ and $\sup_{s\in [0,2]}\norm{q'(s)} \ll 1$.
For the former, we used that $\norm{p(t)} \ll k_0^{-A}$ for all $t<t_0$ and a large absolute constant of our choice (this follows from an earlier assumption that $\norm{r}$ be sufficiently small, as otherwise we conclude trivially from \ref{item:sobineq}).
Now let $n \in \N$ be such that
\begin{align*}
n \in [T^{\frac{1}{\tempered}}, (2T)^{\frac{1}{\tempered}}-1].
\end{align*}
In particular, $n \geq T^{\frac{1}{\tempered}} \gg \norm{r}^{-\star}$.
Also, $n \leq (2 T)^{\frac{1}{\tempered}} \leq \varepsilon^{-1}$.
For any $t,t'\in [n^{\tempered},(n+1)^{\tempered}]$ and $f \in \compactsmooth(X)$ we have
\begin{align*}
f(x_2u_{t'}) 
&= f(x_1 u_{t'} \exp(r + q(t'/T))\\
&= f\big(x_1 u_{t'} \exp(q(t'/T))\big) + O(\norm{r}\Sob_d(f))\\
&= f\big(x_1u_{t'} \exp(q(t/T))\big) + O\big(\big(\norm{r}+ T^{-\frac{1}{\tempered}}\big)\Sob_d(f)\big)
\end{align*}
where we applied \eqref{eq:emv eq 10.1} in the second step and \ref{item:soblip} together with $|t-t'|\ll n^{\tempered-1} \ll T^{1-\frac{1}{\tempered}}$ in the third.
Overall, we obtain by definition of genericity
\begin{align*}
|\mu^{\exp(q(s))}(f)-\mu(f)| \ll \big(\norm{r}+ T^{-\frac{1}{\tempered}}\big)\Sob_d(f) \ll \norm{r}^{\kappa}\Sob_d(f)
\end{align*}
for $s = t/T$ and some absolute $\kappa>0$.
This proves that $\mu$ is $\ll_d \norm{r}^{\kappa}$-almost invariant under $\exp(q(s))$ for $s \in [\frac{1}{2},2]$.

Set $s'= \norm{r}^{\kappa}$ and suppose that $s,s+s' \in [\frac{1}{2},2]$ are such that $\norm{q(s)-q(s+s')} \gg s'$.
As $\mu$ is $\ll_d \norm{r}^{\star}$-almost invariant under $\exp(q(s))$ and $\exp(q(s+s'))$, it is also $\norm{r}^{\star}$-almost invariant under $\exp(q(s+s'))\exp(-q(s))$.
Applying \eqref{eq:emv eq 10.1} for $v_1 = q(s+s')$ and $v_2 = q(s)$ and using $\varepsilon$-almost invariance under $\Fs$ we obtain that 
$\mu$ is $\ll \norm{r}^\kappa$-almost invariant under $\exp(Z)$ for some $Z \in \Fr$ with $\norm{Z} \asymp \norm{r}^{\kappa/2}$.
The proposition now follows from \ref{item:alminv1} rescaling $Z$ to norm $1$.
\end{proof}

\section{Existence of transversal generic points}\label{sec:existence transverse}

The purpose of this section is to establish the existence of two generic points with displacement transversal to an arbitrary intermediate Lie algebra $\Fh \subset \Fs \subsetneq \Fg$. 
Such a pair of points gives rise to additional almost invariance via Proposition~\ref{prop:add inv} (see also \S\ref{sec:intro-addinv}).

\begin{proposition}[Pairs of generic points with transversal displacement]\label{thm:trans2}
Let $d \geq 1$.
There exist constants $\Atransres>1$ (depending on $\dimemb$), $k_0=k_0(d)\geq 1$, $d'>d$, and $\Ctranspts>1$ (depending on $G,\Gamma,U$) with the following property.

Let $\mu$ be the Haar probability measure on a periodic orbit $xH$.
Let $T \geq \Ctranspts$ with
\begin{align}\label{eq:diocondonT}
T \leq \frac{\minobs(xH)^{\frac{1}{\Adiocond}}}{\Ctranspts\minht(xH)}.
\end{align}
Suppose that $\mu$ is $\varepsilon$-almost invariant under an intermediate Lie algebra $\Fg \supsetneq \Fs \supset \Fh$ with respect to $\Sob_{d}$ and let $\Fr$ be an undistorted $H$-invariant complement to $\Fs$.
Then there exist two points $x_1,x_2 \in X$ with the following properties:
\begin{itemize}
\item $x_1,x_2$ are $[k_0,\varepsilon^{-\beta(d)}]$-generic with respect to $\Sob_{d'}$.
\item There exists $r \in \Fr$ so that $x_2= x_1 \exp(r)$ and
\begin{align*}
\kappa T^{-1} \leq \norm{r} \leq T^{-1/\Atransres}
\end{align*}
for some absolute $\kappa>0$.
\end{itemize}
\end{proposition}

We refer to \S\ref{sec:intro-transv} for an outline of the proof which, in particular, relies on the effective closing lemma in the work \cite{EffClosing} with Lindenstrauss, Margulis, Mohammadi and Shah (see also Theorem~\ref{thm:effectiveclosing}) and the effective ergodic theorem in Proposition~\ref{prop:ptsalongU-orbitswith many close gen pts}.
The rest of the section is dedicated to the proof of Proposition~\ref{thm:trans2}.

We fix a `small' compact neighborhood $\Omega \subset \Fs$ of zero to be specified later.
Let $d \geq 1$ be arbitrary and let $d'>d$, $\beta = \beta(d)>0$ be as in Proposition~\ref{prop:ptsalongU-orbitswith many close gen pts}.
The following technical lemma constructs a set of points in $X$ with good properties. 

\begin{lemma}[Existence of a set of `good' points]\label{lemma:Fubini}
There exist $k_0 \geq 1$ depending on $d$ as well as $\eta_0 >0$ depending on $\Gamma$, $U$ with the following property. 

Let $\dioeps = \dioeps_{\eta_0}$ be as defined in \eqref{eq:dioepschoice} and let $T>0$ with
\begin{align}\label{eq:diocondonT2}
\Ceffcloscor\eta_0^{-\Aeffcloscor} \leq T \leq \CeffavcorH^{-1/\Adiocond}\eta_0 \frac{\minobs(xH)^{\frac{1}{\Adiocond}}}{\minht(xH)}.
\end{align}
Then there exists a subset $X_{\mathrm{good}} \subset X$ with $\mu(X_{\mathrm{good}}) \geq \frac{7}{10}$ and with the following properties:
\begin{enumerate}[(i)]
\item  For any $x \in X_{\mathrm{good}}$ the measure of the set of $t \in [0,T]$ such that
\begin{align*}
\vol\big(\big\{Z \in \Omega: xu_t\exp(Z) \text{ is } [k_0,\varepsilon^{-\beta}]\text{-generic w.r.t. } \Sob_{d'}\big\}\big)
< \frac{9}{10} \vol(\Omega)
\end{align*}
is at most $10^{-10}T$.
\item Any $x \in X_{\mathrm{good}}$ is $(\dioeps,T)$-Diophantine.
\item For any $x \in X_{\mathrm{good}}$ 
\begin{align*}
|\{t \in [0,T]: xu_t \not\in X_{\eta_0}\}| \leq 10^{-10}T.
\end{align*}
\end{enumerate}
\end{lemma}

\begin{proof}
The lemma merely combines already established results.
We begin by finding a `large' set of points as in (i).
For any $x \in X$ and $t \in [0,T]$ set
\begin{align*}
f(x,t) = \frac{\vol\big( \{Z \in \Omega: xu_t \exp(Z) \text{ is not } [k_0,\varepsilon^{-\beta}]\text{-generic}\}\big)}{\vol(\Omega)}.
\end{align*}
By Proposition~\ref{prop:ptsalongU-orbitswith many close gen pts} we have
\begin{align*}
\frac{1}{T}\int_{0}^T\int_X f(x,t) \de \mu(x) \de t \ll k_0^{-1}.
\end{align*}
By Chebyshev's inequality, this shows that
\begin{align*}
\frac{1}{T}\int_0^T f(x,t) \de t \leq 10^{-11}
\end{align*}
for $x \in X$ outside of a set of $\mu$-measure $\ll k_0^{-1}$.
For any such $x$, we have after applying Chebyshev's inequality again
\begin{align*}
|\{t \in [0,T]: f(x,t) \geq \tfrac{1}{10}\}| \leq 10^{-10}T.
\end{align*}
Choosing $k_0$ large enough, we can thus ensure the set of points $x \in X$ satisfying (i) has measure at least $\frac{9}{10}$.

We turn to the set of points as in (ii).
Choosing $\eta_0$ small enough (depending on $\Gamma$ and $G$) to ensure that $\Ceffavgeom \eta_0^{\kappaeffav} \leq 10^{-11}$ we obtain that the set of points as in (ii) has measure at least $1-10^{-11}$ by Corollary~\ref{cor:dio on H-orbit} in view of our restrictions on $T$.

For the set of points in (iii) we also apply Corollary~\ref{cor:dio on H-orbit} to obtain that $\mu(X \setminus X_{\eta_0}) \leq 10^{-11}$.
This implies by $U$-invariance of $\mu$
\begin{align*}
\int_X \frac{|\{t \in [0,T]: xu_t \not \in X_{\eta_0}\}|}{T} \de \mu(x) \leq 10^{-11}.
\end{align*}
In particular, the set of points $x \in X$ with $|\{t \in [0,T]: xu_t \not \in X_{\eta_0}\}| \geq 10^{-10}T$ has measure at most $\frac{1}{10}$.

We have shown that the set of points as in (i), (ii), and (iii) has $\mu$-measure at least $\frac{7}{10}$ which proves the lemma.
\end{proof}

\begin{proof}[Proof of Proposition~\ref{thm:trans2}]
Throughout the proof of Proposition~\ref{thm:trans2}, we shall take $k_0$, $\eta_0$, $\dioeps$, and $X_{\mathrm{good}}$ to be as in Lemma~\ref{lemma:Fubini}. 
Moreover, $T>0$ is assumed to be larger than a fixed large constant and is taken to satisfy \eqref{eq:diocondonT2}.

Fix a point $x \in X_{\mathrm{good}}$ and write $\mathcal{T}\subset [0,T]$ for the set of times $t \in [0,T]$ with the following properties:
\begin{itemize}
\item $xu_t \in X_{\eta_0}$.
\item The set of $Z \in \Omega$ such that $xu_t\exp(Z)$ is $[k_0,\varepsilon^{-\beta}]$-generic has measure at least $\frac{9}{10}\vol(\Omega)$.
\end{itemize}
By Lemma~\ref{lemma:Fubini}, $|\mathcal{T}| \geq (1-10^{-9})T$.

Fix small neighborhoods $\Omega_\Fs\subset \Omega$ (resp.~$\Omega_\Fr$) of zero in $\Fs$ (resp.~$\Fr$).
Each of these neighborhoods is assumed to be a small ball around zero with respect to $\norm{\cdot}$ where the radius is chosen independently of $\Fs$ or $\Fr$.
In particular, we may assume that
\begin{align*}
(Y_1,Y_2) \in \Omega_\Fs \times \Omega_\Fr \mapsto y\exp(Y_1)\exp(Y_2)
\end{align*}
is a diffeomorphism onto its image where the derivative is uniformly close to the identity independently of $\Fs$ and $\Fr$ (using the same argument as in the proof of Lemma~\ref{lem:EMV10.1} based on $\Fr$ being undistorted).
Furthermore, we may suppose that for any $y \in X_{\eta_0}$ the map $g\in\exp(\Omega_\Fs)\exp(\Omega_\Fr) \mapsto yg$ is bijective (by compactness of  $X_{\eta_0}$).
We denote for any $y \in X_{\eta_0}$
\begin{align*}
\phi_y: (Y_1,Y_2) \in \Omega_\Fs \times \Omega_\Fr \mapsto y\exp(Y_1)\exp(Y_2).
\end{align*}
By the above, $\phi_y$ is a diffeomorphism onto its image $\Omega(y) := y\exp(\Omega_\Fs)\exp(\Omega_\Fr)$ where the derivative is uniformly close to the identity independently of $\Fs$ and $\Fr$.
By the pigeonhole principle, we may fix $y\in X_{\eta_0}$ such that
\begin{align}\label{eq:estimate calT}
|\{t \in \mathcal{T}: xu_t \in \Omega(y) \}| \gg T.
\end{align}
Write $\phi=\phi_y$ for simplicity.

Fix $\kappa \in (0,1)$ to be determined.
We set for all $R >0$ and $Y \in \Omega_\Fs+ \Omega_\Fr$
\begin{align*}
\Omega_\Fs(R) &:= \{Y \in \Omega_\Fs: \norm{Y} \leq R^{-\kappa}\},\\
\Omega_\Fr(R) &:= \{Y \in \Omega_\Fr: \norm{Y} \leq R^{-2}\},\\
B_Y(R) &:= Y +\Omega_\Fs(R)+\Omega_\Fr(R).
\end{align*}
There is a constant $\mathrm{C}>1$ such that if $\Omega_\Fs,\Omega_\Fr$ are small enough, we have that for any $Y \in \Omega_\Fs+ \Omega_\Fr$
\begin{align}\label{eq:elprop thinbox1}
\phi(B_Y(R)) \subset y'\exp(\Omega_\Fs(\mathrm{C}R)) \exp(\Omega_\Fr(\mathrm{C}R)) =: \Omega(y',\mathrm{C}R)
\end{align} 
for $y' = y \exp(Y)$. Furthermore, we may assume that for any $y'' \in \Omega(y',R)$
\begin{align}\label{eq:elprop thinbox2}
\Omega(y',R) \subset \Omega(y'',\mathrm{C}R).
\end{align}

\begin{claim*}
The set of $t \in \mathcal{T}$ with $xu_t \in \phi(B_Y(T))$ has measure at most $T^{1-1/(2\Aeffcloscor)}$ for all $Y \in \Omega_\Fs+ \Omega_\Fr$.
\end{claim*}

\begin{proof}[Proof of Claim]
We use the effective closing lemma in the form of Corollary~\ref{cor:effective closing}.
Suppose that there exists $Y \in \Omega_\Fs+ \Omega_\Fr$ so that the set $\mathcal{E}$ of times $t \in \mathcal{T}$ with $xu_t \in \phi(B_Y(T))$ has measure at least $T^{1-1/(2\Aeffcloscor)}$.
For any $s,t \in \mathcal{E}$ we have $xu_{s},xu_{t} \in \Omega(y\exp(Y),\mathrm{C}T)$ by \eqref{eq:elprop thinbox1}. 
Hence, $xu_{s} \in \Omega(xu_{t},\mathrm{C}^2T)$ by \eqref{eq:elprop thinbox2} and there exist $X^{\Fr}_{s,t}\in \Fr$, $X^{\Fs}_{s,t}\in\Fs$ so that
\begin{align*}
xu_s = xu_t\exp(X^{\Fs}_{s,t})\exp(X^{\Fr}_{s,t})
\end{align*}
with
\begin{align*}
\norm{X^{\Fr}_{s,t}} \leq (\mathrm{C}^2T)^{-2},\quad
\norm{X^{\Fs}_{s,t}} \leq (\mathrm{C}^2T)^{-\kappa}.
\end{align*}
Setting $g_{s,t}= \exp(X^{\Fs}_{s,t})\exp(X^{\Fr}_{s,t})$, Corollary~\ref{cor:effective closing} applies whenever $T$ is sufficiently large.
But note that $\Fs$ is not an ideal (cf.~Remark~\ref{rem:notnormal}) which yields a contradiction and proves the claim.
\end{proof}

Take a cover of $\Omega(y)$ by `boxes' of the form $\phi(B_Y(T))$ with $O(1)$ overlaps.
The claim together with \eqref{eq:estimate calT} implies that $\gg T^{1/(2\Aeffcloscor)}$ of the boxes in that cover contain a point of the form $xu_t$ for $t \in \mathcal{T}$.
Note that the partitioning is chosen not to be `too fine' in the directions of $\Fs$.
Indeed, $\Omega_\Fs$ is covered by $\ll T^{\kappa\dim(\Fs)}$ translates of $\Omega_{\Fs}(T)$.
Thus, for $\gg T^{1/(2\Aeffcloscor)-\kappa\dim(\Fs)}$ many boxes in the $\Fr$ direction --- that is, translates of $\Omega_{\Fr}(T)$ --- there exists a box in the $\Fs$-direction so that the corresponding joint box, say $\phi(B_Y(T))$, contains a point of the form $xu_t$ for $t \in \mathcal{T}$.
By the pigeonhole principle, we can find two times $t_1,t_2$ such the displacement between  $xu_{t_1}$ and $xu_{t_2}$ in the $\Fr$-direction is at most $T^{-\frac{1}{2\dim(\Fr)} (1/(2\Aeffcloscor)-\kappa\dim(\Fs))}$.
More explicitly, there are $t_1,t_2 \in \mathcal{T}$ and $Y_1 = Y_1^\Fs + Y_1^\Fr$, $Y_2 = Y_2^\Fs+ Y_2^\Fr$ such that the points
\begin{align*}
x_1 &= xu_{t_1} = y \exp(Y_1^\Fs) \exp(Y_1^\Fr),\\
x_2 &= xu_{t_2} = y \exp(Y_2^\Fs) \exp(Y_2^\Fr)
\end{align*}
satisfy
\begin{align*}
T^{-1} \ll \norm{Y_1^\Fr-Y_2^\Fr} 
\leq T^{-\frac{1}{2\dim(\Fr)}(1/(2\Aeffcloscor)-\kappa\dim(\Fs))}.
\end{align*}
We take $\kappa = (4\Aeffcloscor\dim(\Fg))^{-1}$.

As in \cite[\S14]{emv}, we now perturb $x_1,x_2$ in the $\Fs$-direction. For $i=1,2$ we denote by $W_i$ the set of $Z \in \Omega$ such that $x_i\exp(Z)$ is $[k_0,\varepsilon^{-\beta}]$-generic.
By construction, we have $\vol(W_i) \geq \frac{9}{10} \vol(\Omega)$.
As these sets $W_i$ have sufficiently large measure, one may argue as in \cite[p.~195]{emv} to find $Z_i \in W_i$ for $i=1,2$ such that
\begin{align*}
x_2 \exp(Z_2) = x_1 \exp(Z_1) \exp(Y)
\end{align*}
for some $Y \in \Fr$ with
\begin{align*}
T^{-1} \ll \norm{Y} \ll T^{-\frac{1}{2\dim(\Fr)}(1/(2\Aeffcloscor)-\kappa\dim(\Fs))}
\end{align*}
(assuming that $\Omega_{\Fs},\Omega,\Omega_{\Fr}$ are sufficiently small). 
Thus, the points $x_1 \exp(Z_1)$ and $x_2\exp(Z_2)$ satisfy the requirements of the proposition and we conclude.
\end{proof}

\section{Proof of Theorem~\ref{thm:main}}\label{sec:proofmain}

We prove Theorem~\ref{thm:main} by a simple induction process:
Starting with $\Fs_0 = \Fh$ we can apply Proposition~\ref{thm:trans2} to obtain two generic points with displacement transversal to $\Fh$. 
Using these two points we obtain that $\mu$ is almost invariant under a larger subalgebra $\Fs_1 \supsetneq \Fh$ by Proposition~\ref{prop:add inv} and Proposition~\ref{prop:effective generation}.
We repeat this procedure using $\Fs_1$ (where we recall that no information on $\Fs_1$ is needed other than it contains $\Fh$). 
After at most $\dim(G)$ steps, one obtains that $\mu$ is almost invariant under $\Fg$.
(Note that the order (or degree) of the Sobolev norm $\Sob_d$ needs to be increased at every step.)
The following proposition then implies the theorem.

\begin{proposition}[{\cite[Prop.~15.1]{emv}}]\label{prop:alminv under G}
There is $\kappaalminvG>0$ with the following property. 
Let $\mu_{xG^+}$ be the $G^+$-invariant probability measure on $xG^+ \subset X$.
Suppose that $\mu$ is a probability measure on $xG^+$ which is $\varepsilon$-almost invariant under $\Fg$ with respect to $\Sob_d$ for some $d\geq 1$. Then
\begin{align*}
|\mu_{xG^+}(f) -\mu(f)|\ll_{\Gamma,d} \varepsilon^{\kappaalminvG/d}\Sob_d(f).
\end{align*}
\end{proposition}

\begin{proof}[Proof of Theorem~\ref{thm:main}]
We may assume that
\begin{align*}
\mathrm{C}\,\minht(xH) \leq \minobs(xH)^{\frac{1}{2\Adiocond}}
\end{align*}
for some large constant $\mathrm{C}>0$ where the notation $\minobs(xH)$ was introduced in \eqref{eq:ethdef}.
Indeed, the theorem follows easily from \ref{item:sobineq} otherwise.

We proceed recursively as outlined above. Start with $\Fs_0 = \Fh$ so that $\mu$ is $\varepsilon$-almost invariant under $\Fs_0$ for any $\varepsilon>0$ by definition (for any choice of Sobolev norm, say $\Sob_{d_0}$).
Let $\Fr$ be an undistorted $H$-invariant complement to $\Fh$.
Apply Proposition~\ref{thm:trans2} with
\begin{align*}
T = \minobs(xH)^{\frac{1}{2\Adiocond}}.
\end{align*}
Thus, there exists $d_1 \geq d_0$ and two $[k_0(d_0),\varepsilon^{-\beta(d_0)}]$-generic points $x_1,x_2$ (with respect to $\Sob_{d_1}$) with $x_2 = x_1 \exp(r)$ for some $r \in \Fr$ such that
\begin{align*}
T^{-\Atransres} \leq \norm{r} \leq T^{-1/\Atransres}.
\end{align*}
By Proposition~\ref{prop:add inv}, $\mu$ is $\ll T^{-\star}$-almost invariant under some unit vector $Z \in \Fr$ (with respect to $\Sob_{d_1}$).
Therefore, by Proposition~\ref{prop:effective generation} the measure $\mu$ is $ \varepsilon_1$-almost invariant under a Lie subalgebra $\Fs_{1} \supset \Fh$ (with respect to $\Sob_{d_1}$) of dimension $\dim(\Fs_1) > \dim(\Fh)$ for $\varepsilon_1 \ll T^{-\star}\ll \minobs(xH)^{-\star}$.

One can proceed in this manner recursively; we limit ourselves to the second step for notational simplicity.
Let $\Fr_1$ be an undistorted $H$-invariant complement to $\Fs_1$. Let $T >0$ be sufficiently large satisfying
\begin{align*}
T \leq \minobs(xH)^{\frac{1}{2\Adiocond}}.
\end{align*}
By Proposition~\ref{thm:trans2} there exists $d_2 \geq d_1$ and two $[k_0(d_1),\varepsilon_1^{-\beta(d_1)}]$-generic points $x_1,x_2$ (with respect to $\Sob_{d_2}$) with $x_2 = x_1 \exp(r)$ for some $r \in \Fr_1$ such that
\begin{align*}
T^{-\Atransres} \leq \norm{r} \leq T^{-1/\Atransres}.
\end{align*}
We choose $T$ to be a sufficiently small power of $\varepsilon_1^{-1}$ so that \eqref{eq:addinv-displcond} is satisfied.
By Proposition~\ref{prop:add inv}, $\mu$ is $\ll \varepsilon_1^{\star}$-almost invariant under some unit vector $Z\in\Fr_1$.
By Proposition~\ref{prop:effective generation} the measure $\mu$ is $ \varepsilon_2$-almost invariant under a Lie subalgebra $\Fs_{2} \supset \Fh$ of dimension $\dim(\Fs_2) > \dim(\Fs_1)$ for $\varepsilon_2 \ll \varepsilon_1^{\star} \ll \minobs(xH)^{-\star }$.
Continuing like this proves that $\mu$ is $\ll \minobs(xH)^{-\star}$-almost invariant under $\Fg$ and the theorem follows from Proposition~\ref{prop:alminv under G}.
\end{proof}

\section{Extensions}\label{sec:exhaustive}

A priori, the function $\minobs(xH)$ measuring the minimal complexity of all intermediate orbits for a periodic orbit $xH$ is hard to estimate (note that Theorem~\ref{thm:main} is equivalently phrased with a polynomial rate in $\minobs(xH)$). 
Indeed, it appears to require a classification of all subgroups of class $\classH$ containing $gHg^{-1}$ for $g \in G$ with $\Gamma gH$ periodic.
In the following, we aim to reduce the set of intermediate groups.

\begin{definition}\label{def:exhaustive}
Let $\BM$ be a $\Q$-subgroup of $\G$ and let $\mathrm{A}>1$, $\mathrm{C}>1$.
We say that a subcollection $\classH' \subset \classH$ is $(\mathrm{A},\mathrm{C})$-exhaustive for $\BM$ if the following property holds: 
whenever $\BM$ is contained in a proper $\Q$-subgroup $\BL \in \classH$ of $\G$ with $\height(\BL) \leq T$ then $\BM$ is also contained in a proper $\Q$-subgroup $\BL' \in \classH'$ with $\height(\BL') \leq \mathrm{C}T^{\mathrm{A}}$.
A subcollection $\classH' \subset \classH$ is exhaustive for $\BM$ if it is $(\mathrm{A},\mathrm{C})$-exhaustive for some $\mathrm{A}>1$, $\mathrm{C}>1$.
\end{definition}

We shall give concrete examples of exhaustive subcollections in \S\ref{sec:exhaustive} below.
If $xH= \Gamma gH$ is a periodic orbit in $X$, the Borel density theorem \cite{Borel-density} shows that $\Gamma \cap gHg^{-1}$ is Zariski dense in $gHg^{-1}$.
In particular, the Zariski-closure of $gHg^{-1}$ is defined over $\Q$.
This shows that any periodic orbit $xH$ gives rise to a $\Gamma$-conjugacy class of semisimple $\Q$-groups, where the $\Gamma$-conjugacy class removes the ambiguity of the choice of representative $x = \Gamma g$.
We say that a subcollection  $\mathcal{H}'\subset \mathcal{H}$ is $(\mathrm{A},\mathrm{C})$-exhaustive for $xH$ if it is $(\mathrm{A},\mathrm{C})$-exhaustive for any element of that conjugacy class.
Lastly, we define in analogy to \eqref{eq:ethdef}
\begin{align*}
\minobs_{\classH'}(xH) = \inf\{\disc(\Gamma \BM(\R) g): \G \neq \mathbf{M} \in \classH' \text{ with } xH \subset \Gamma \BM(\R) g\}.
\end{align*}

The following theorem extends Theorem~\ref{thm:main}.

\begin{theorem}\label{thm:exhaustive}
There exists $d \geq 1$ depending on $G$ and $H$ with the following properties.
Let $xH = \Gamma g H$ be a periodic orbit and let $\classH'\subset \classH$ be an $(\mathrm{A},\mathrm{C})$-exhaustive $\Gamma$-invariant subcollection for $xH$.
Then there exists $\delta>0$ depending only on $\mathrm{A},G,H$ such that
\begin{align*}
\Big| \int_{xH} f \de \mu_{xH} - \int_{xG^+} f \de \mu_{xG^+}\Big| 
\ll \minobs_{\classH'}(xH)^{-\delta}\minht(xH) \Sob_d(f)
\end{align*}
for all $f \in C_c^\infty(X)$.
Here, the implicit constant depends on $G$, $H$, $\Gamma$, and $\mathrm{C}$.
\end{theorem}

\begin{proof}
By Definition~\ref{def:exhaustive}, we have
\begin{align*}
\minobs(xH) \leq \minobs_{\classH'}(xH) \ll_{\mathrm{C}} \minht(xH)^\star \minobs(xH)^\star.
\end{align*}
We may assume that $\minht(xH) \leq \theta_{\mathcal{H}'}(xH)^{\delta'}$ for a fixed $\delta'>0$ (else the theorem follows from \ref{item:sobineq}).
Thus, 
\begin{align*}
\Big| \int_{xH} f \de \mu_{xH} - \int_{xG^+} f \de \mu_{xG^+}\Big| 
&\ll \minobs_{\classH}(xH)^{-\delta}\minht(xH) \Sob_d(f)\\
&\ll \minobs_{\classH'}(xH)^{-\star\delta}\minht(xH)^{\star}\Sob_d(f)
\end{align*}
implies Theorem~\ref{thm:exhaustive} by choosing $\delta'>0$ sufficiently small.
%
\end{proof}

\subsection{Examples of exhaustive subcollections}\label{sec:examples-exhaustive}
In this subsection we shall give a few examples of exhaustive subcollections with a focus on the setup in Theorem~\ref{thm:effequi orth}.

\subsubsection{General examples}
Given any $\Q$-subgroup $\BM < \G$ we write $\BN_\BM$ for the normalizer subgroup and $\BM^{\mathcal{H}}$ for the largest subgroup of $\BM$ which has class $\classH$.
Then
\begin{align*}
\height(\BN_\BM),\height(\BM^{\classH}) \ll \height(\BM)^\star;
\end{align*}
see \cite[\S4]{effectiveavoidance}.
The following lemma is elementary.

\begin{lemma}\label{lem:exhaustiveexp}
For any non-normal subgroup $\BH \in \classH$ of $\G$ there exists a connected non-normal $\Q$-subgroup $\BM \subset \G$ containing $\BH$ with $\height(\BH) \ll \height(\BM)^\star$ and $\BN_\BM^ \circ = \BM$ so that one of the following is true:
\begin{enumerate}[(i)]
\item $\BM$ is a parabolic $\Q$-group.
\item $\BM$ is the centralizer of a $\Q$-torus.
\item $\BM$ is semisimple with finite centralizer.
\end{enumerate}
In particular, the collection of subgroups $\BM^{\classH}$ for $\BM$ as in (i), (ii), or (iii) together with all normal subgroups of $\G$ is an exhaustive subcollection for $\BH$.
\end{lemma}

We remark that a centralizer of a $\Q$-torus $\BT$ is always a Levi subgroup of a parabolic subgroup defined over $\overline{\Q}$ but in general such a parabolic subgroup cannot be chosen to be defined over $\Q$.

\begin{proof}
Suppose first that the unipotent radical $\BU_\BH$ of $\BH$ is non-trivial. We define $\BM_0 = \BH$ and inductively $\BM_{j} = \BN_{\BU_{\BM_{j-1}}}^\circ$ for $j \geq 1$. We obtain an increasing sequence of connected $\Q$-subgroups
\begin{align*}
\BH = \BM_0 \subset \BM_1 \subset \BM_2 \subset \ldots
\end{align*}
The sequence needs to stabilize after at most $\dim(\G)$ steps and hence we obtain a connected $\Q$-subgroup $\BM$ containing $\BH$ with  $\height(\BH) \ll \height(\BM)^\star$ and $\BM = \BN_{\BU_{\BM}}^\circ$.
Such a subgroup must be parabolic \cite{BorelTits-unipotents}.

Suppose now that $\BU_\BH$ is trivial and consider the nested sequence defined by $\BM_0 = \BH$ and $\BM_j = \BN_{\BM_{j-1}}^\circ$. If $\BM_j$ has non-trivial unipotent radical for some $j$, one can apply the previous step\footnote{Alternatively, one can use the fact that the centralizer of a reductive $\Q$-group is always reductive.
}.
Also, note that $\BM_j$ is never a normal subgroup of $\G$. Indeed, otherwise $\BM_i$ is a product of some of the simple factors of $\G$ for every $i \leq j$ by backwards induction contradicting the non-normality assumption on $ \BH$.
If $\BM_j$ has trivial unipotent radical for every $j$, there is a connected reductive $\Q$-group $\BM$ containing $\BH$ with $\height(\BM) \ll \height(\BH)^\star$ and with $\BM = \wholenorm{\BM}^\circ$.
If the center $\BZ$ of $\BM$ is finite, we are done.
Otherwise, the centralizer\footnote{Note that $\BM'$ is not necessarily equal to $\BM$. For example, take $\BM$ to be the product of the block-diagonally embedded copy of $\SL_2$ in $\SL_4$ with its centralizer.} $\BM'$ of $\BZ$ contains $\BM$ and satisfies $\height(\BM') \ll \height(\BH)^\star$.
This concludes the lemma.
\end{proof}

%
%

When $\G$ is given as product of $\Q$-almost simple groups, one could wish for an exhaustive subcollection of subgroups in product form. Theorem~\ref{thm:exhaustive} in this case is particularly useful when trying to establish disjointness. For simplicity, we only state the following lemma for two factors; as it is not used in the remainder of the article, we omit the proof.

\begin{lemma}
Suppose that $\G = \G_1 \times \G_2$ where $\G_1, \G_2$ are $\Q$-almost simple simply connected $\Q$-groups. Consider the subcollection $\classH' \subset \classH$ of $\Q$-groups $\BM< \G$ of one of the following two forms:
\begin{itemize}
\item $\BM = \BM_1 \times \G_2$ for $\BM_1 < \G_1$ or $\BM = \G_1 \times \BM_2$ for $\BM_2 < \G_2$.
\item $\BM$ is the graph of an isomorphism $\G_1 \to \G_2$.
\end{itemize}
Then $\classH'$ is exhaustive for any $\BH \in \classH$.
\end{lemma}
%
%

\subsubsection{Special orthogonal groups}

The aim of this section is to determine an exhaustive subcollection used in the proof of Theorem~\ref{thm:effequi orth}. 
For that purpose, let $Q$ be a nondegenerate\footnote{Here, $Q$ is non-degenerate if there is no non-zero $v \in \Q^n$ with $(v,w)_Q = 0$ for all $w \in \Q^n$.} rational quadratic form on $\Q^n$ for $n \geq 4$ and let $\G = \SO_Q$.
For any subset $\mathcal{B}\subset \Q^n$ we set 
\begin{align*}
\BM_{\mathcal{B}} = \{g\in \SO_Q: g.w = w \text{ for all }w \in \mathcal{B}\}.
\end{align*}
Whenever $L \subsetneq \Q^n$ is a subspace with $Q|_L$ non-degenerate, the group $\BM_L\simeq \SO_{Q|_L}$ is semisimple if $\dim(L)\leq n-3$ and is a torus if $\dim(L) = n-2$.
It is a maximal (connected) subgroup if and only if $\dim(L)=1$ (this goes back to Dynkin \cite{Dynkin1,Dynkin2}).

\begin{proposition}[Pointwise stabilizer subcollections]\label{prop:exhaustive orth}
Let $L \subset \Q^n$ be a subspace of dimension at most $n-3$. Assume that $Q|_L$ is positive definite over $\R$.
The subcollection
\begin{align*}
\mathcal{H}_L' = \{\BM_v: 0 \neq v \in L\}
\end{align*}
is $(\mathrm{A},\mathrm{C})$-exhaustive for $\BM_L$ for some $(\mathrm{A},\mathrm{C})$ depending only on $n$.
\end{proposition} 

We will use the following lemma.

\begin{lemma}[Heights of normal subgroups]\label{lem:normalsubgroups}
Let $\BM < \SL_{\dimemb}$ be a semisimple $\Q$-subgroup. Then $\height(\BM') \ll_{\dimemb} \height(\BM)^\star$ for any normal $\Q$-subgroup $\BM' \triangleleft \BM$.
\end{lemma}

\begin{proof}
We realize any $\Q$-simple Lie ideal $\Fm' \triangleleft \Fm$ by integral linear equations with controlled coefficients.
To that end, let $\BT$ be the identity component of the centralizer of $\Ad(\BM)$ in $\SL(\Fm)$ where we use coordinates induced by a basis of $\Fm$ of integral vectors of norm $\ll \height(\Fm)^\star$ (which exists by Minkowski's second theorem).
Any Lie ideal $\Fm' \triangleleft \Fm$ (defined over $\overline{\Q}$) is $\BT$-invariant and, 
if $\Fm'$ is absolutely simple, the action of $\BT$ on $\Fm'$ is by scalars.
As this holds for all $\mathfrak{m}'$ and $\mathfrak{m}$ is a direct sum of its simple ideals (over $\overline{\Q}$), $\BT$ is a torus.
Conversely, if $\Fm_\alpha$ is a weight space for a weight $\alpha$ of $\Ft$, then $\Fm_\alpha$ is an absolutely simple ideal.

Let $\mathcal{A}$ be the set of weights of the $\Ft$-representation on $\Fm$.
The absolute Galois group $\mathcal{G} = \Gal(\overline{\Q}/\Q)$ acts on $\mathcal{A}$ and we set $\Fm_{\mathcal{G}.\alpha} = \bigoplus_{\alpha' \in \mathcal{G}.\alpha}\Fm_{\alpha'}$ for $\alpha \in \mathcal{A}$.
By construction, $\Fm_{\mathcal{G}.\alpha}\subset \Fm$ is a $\Q$-simple Lie ideal for any $\alpha \in \mathcal{A}$  and any $\Q$-simple Lie ideal is of this form.

Taking direct sums, it suffices to bound the height of $\Fm_{\mathcal{G}.\alpha}$ for any $\alpha$.
Thus, fix $\alpha \in \mathcal{A}$ and note that
\begin{align*}
\Fm_{\mathcal{G}.\alpha} = \Big\{v \in \Fm: \prod_{\alpha' \in \mathcal{G}.\alpha}(X-\alpha'(X))v = 0 \text{ for all }X \in \Lie(\BT)(\Q)\Big\}.
\end{align*}
By construction, we have $\height(\BT) \ll \height(\BM)^\star$ and, by Minkowski's second theorem, we may fix an integral basis $X_1,\ldots, X_t$ of $\Lie(\BT)$ with $\norm{X_i}\ll \height(\BM)^\star$.
Given integers $a_1,\ldots, a_t$, $X= \sum_i a_i X_i$, and $\beta \in \mathcal{A}$ the expression
\begin{align*}
P_\beta(a_1,\ldots,a_t) =
\prod_{\alpha' \in \mathcal{G}.\alpha}(\beta(X)-\alpha'(X))
= \prod_{{\alpha' \in \mathcal{G}.\alpha}} \sum_i a_i (\beta(X_i)-\alpha'(X_i))
\end{align*}
is a polynomial in $a_1,\ldots,a_t$ of (total) degree at most $\#(\mathcal{G}.\alpha)$.
It is non-zero if and only if $ \beta\notin\mathcal G.\alpha$.
Let $P = \prod_{ \beta\notin\mathcal G.\alpha} P_\beta$ and let $d$ be its degree. 
We can fix integers $a_1,\ldots, a_t$ with $|a_i| \leq d$ so that $P(a_1,\ldots,a_t) \neq 0$ or equivalently $P_\beta(a_1,\ldots,a_t) \neq 0$ for all $\beta\notin\mathcal G.\alpha$.
Setting $X = \sum_i a_i X_i$, we deduce
\begin{align*}
\Fm_{\mathcal{G}.\alpha} = \Big\{v \in \Fm: \prod_{\alpha' \in \mathcal{G}.\alpha}(X-\alpha'(X))v = 0\Big\}.
\end{align*}
For every $\alpha'$ the eigenvalue $\alpha'(X)$ of $X$ is an algebraic integer and satisfies $|\alpha'(X)|\ll \norm{X}^\star \ll \height(\BM)^\star$.
Thus, $\Fm_{\mathcal{G}.\alpha}$ is defined by integral linear equations with coefficients of size $\ll \height(\BM)^\star$.
This implies the lemma.
\end{proof}

\begin{proof}[Proof of Proposition~\ref{prop:exhaustive orth}]
Let $\BM<\G$ be a $\Q$-group of class $\classH$ containing $\BM_L$. 
If $n =4$, $\BM_L$ is maximal and $\BM_L = \BM$ i.e.~there is nothing to prove. 
So we assume $n \geq 5$ throughout and, in particular, $\G = \SO_Q$ is absolutely almost simple.
By Lemma~\ref{lem:exhaustiveexp} we may assume that $\BM$ is a subgroup as in (i)--(iii) of that lemma.

We begin first by observing that (i) cannot occur. Indeed, the centralizer of $\BM_L$ is equal to $\BM_{L^\perp} \simeq \SO_{Q|_L}$ (up to finite index) and hence $\Q$-anisotropic (as $L$ is assumed positive definite).

Suppose now that $\BM$ is the centralizer of a $\Q$-torus $\BT$ as in (ii). In particular, $\BT$ commutes with $\BM_L$ which implies that $\BT$ preserves $L$ and the orthogonal complement $L^\perp$. In fact, since $\BM_L$ acts irreducibly on $L^\perp$, $\BT$ fixes $L^\perp$ pointwise. The subspace
\begin{align*}
V = \{v \in \Q^n: t.v= v\text{ for all } t \in \BT(\Q)\}
\end{align*}
thus contains $L^\perp$ and the restriction of $Q$ to $V^\perp \subset L$ is positive-definite over $\R$. 
Since $\height(\BT) \ll \height(\BM)^\star$, $V$ and $V^\perp$ are defined by integral linear equations with coefficients of size $\ll\height(\BM)^\star$. 
In particular, by Siegel's lemma there exists a nonzero integral vector $v\in V^\perp \subset L$ with $\norm{v}\ll \height(\BM)^\star$.
By construction, $\height(\BM_v)\ll \height(\BM)^\star$ and $\BM_v \supset \BM_L$ which proves the proposition in this case.

Lastly, suppose that $\BM\supset \BM_L$ is semisimple with finite centralizer as in (iii). The following argument roughly follows \cite[\S2.6]{localglobalEV}.
Let $\BM'\subset \BM$ be the subgroup generated by all $\BM$-conjugates of $\BM_L$. It is a normal subgroup defined over $\Q$ and hence $\height(\BM') \ll \height(\BM)^\star$ by Lemma~\ref{lem:normalsubgroups}.
We may assume henceforth that $\BM = \BM'$.
Assume first that $\BM$ does not act irreducibly in the standard representation. 
If $V$ is a nontrivial $\BM$-invariant subspace and $m \in \BM(\Q)$, then it is also $m\BM_L m^{-1}= \BM_{m.L}$-invariant.
In particular, either $V \subset m.L$ or $V \supset (m.L)^\perp$.
As $\BM$ is connected, one of these options occurs for all $m \in \BM(\Q)$; after taking orthogonal complements we may assume the first.
Thus, $V \subset V' = \bigcap_{m \in \BM(\Q)} m.L$. 
Hence, the right-hand side $V'$ is a nontrivial $\BM$-invariant subspace and 
\begin{align*}
V' = \{v \in \Q^n: m.v = v \text{ for all }m \in \BM(\Q)\}.
\end{align*}
Hence the argument follows as in the previous case.

Lastly, assume that $\BM$ acts irreducibly.
An argument of Ellenberg and Venkatesh in \cite[\S2.6]{localglobalEV} based on classification results of Guralnick and Saxl \cite{GuralnickSaxl} for groups generated by reflections implies that $\BM = \G$.
Alternatively, one may invoke the more elementary argument in \cite[Prop.~3.1]{EinsiedlerWirth} which classifies all real intermediate Lie algebras between $\mathfrak{so}(p,q)$ and $\mathfrak{so}(p+r,q)$ for any $p\geq 2$ and $q,r \geq 1$.
See also the arguments of Lee and Oh in \cite[\S3]{LeeOh}.
The proposition follows.
\end{proof}

\begin{proof}[Proof of Theorem~\ref{thm:effequi orth}]
The proof follows directly from Proposition~\ref{prop:exhaustive orth}, Theorem~\ref{thm:exhaustive}, and the fact that $\minht(xH) \ll 1$.
The latter follows from compactness of the centralizer by work of Dani and Margulis \cite{DaniMargulis91} (see also \cite[Cor.~6.4]{effectiveavoidance}).
\end{proof}

\bibliographystyle{amsplain}
\bibliography{Bibliography}

\end{document}